\documentclass[11pt,amsfonts, epsfig]{amsart}

\usepackage{amsmath, amscd, amssymb}
\usepackage[frame,cmtip,arrow,matrix,line,graph,curve]{xy}
\usepackage{epsfig}
\usepackage{graphpap, color}
\usepackage[mathscr]{eucal}
\usepackage{mathrsfs}

\usepackage{pstricks}
\usepackage{color}
\usepackage{cancel}

\numberwithin{equation}{section}

\def\sO{{\mathscr O}}
\def\sM{{\mathscr M}}

\def\sL{{\mathscr L}}

\def\sO{\mathscr{O}}

\def\sR{\mathscr{R}}

\newcommand{\PP}{\mathbb{P}}

\newcommand{\ZZ}{\mathbb{Z}}

\newcommand{\bk}{\mathbf{k}}

\newcommand{\kk}{\bk}

\newcommand{\bP}{\mathbf{P}}

\newcommand{\cal}{\mathcal}

\def\cA{{\cal A}}
\def\cB{{\cal B}}
\def\cC{{\cal C}}
\def\cD{{\cal D}}
\def\cE{{\cal E}}
\def\cF{{\cal F}}

\def\cO{{\cal O}}

\def\cU{{\cal U}}
\def\cV{{\cal V}}

\def\cS{{\cal S}}
\def\cX{{\cal X}}

\def\fD{\mathfrak{D}}

\def\fE{\mathfrak{E}}

\def\fM{\mathfrak{M}}

\def\mapright#1{\,\smash{\mathop{\lra}\limits^{#1}}\,}

\def\dual{^{\vee}}

\def\sta{^\ast}

\def\upmo{^{-1}}
\def\sta{^{\ast}}

\def\pri{^{\prime}}

\def\sta{^*}

\def\lra{\longrightarrow}

\def\lsta{_{\ast}}

\def\Oplus{\mathop{\oplus}}

\def\begeq{\begin{equation}}
\def\endeq{\end{equation}}
\def\and{\quad{\rm and}\quad}
\def\bl{\bigl(}
\def\br{\bigr)}

\def\mh{\!:\!}
\def\sub{\subset}
\def\Ao{{\mathbb A}^{\!1}}

\def\and{\quad\text{and}\quad}
\def\mapright#1{\,\smash{\mathop{\lra}\limits^{#1}}\,}

 \DeclareMathOperator{\Ext}{Ext}

\let\lab=\label

\newtheorem{prop}{Proposition}[section]
\newtheorem{theo}[prop]{Theorem}
\newtheorem{lemm}[prop]{Lemma}
\newtheorem{coro}[prop]{Corollary}

\newtheorem{exam}[prop]{Example}
\newtheorem{defi}[prop]{Definition}

\theoremstyle{definition}
\newtheorem{say}[prop]{}

\def\bone{{\mathbf 1}}

\def\Pn{{\mathbb P}^n}
\def\PP{{\mathbb P}}

\def\AA{{\mathbb A}}

\def\MPd{\overline M_1(\Pn,d)}

\def\wMPd{\widetilde{M}_1(\Pn,d)}

\def\fMw{\fM\uwt_1}
\def\wfMw{\widetilde{\fM}\uwt_1}

\def\sta{^\ast}

\let\lab=\label

\def\sO{{\mathscr O}}

\def\lab#1{\label{#1}[{#1}]\  }

\def\lab{\label} 

\def\beq{\begin{equation}}
\def\eeq{\end{equation}}

\def\pri{^{\prime}}

\let\ga=\gamma
\let\eset=\emptyset
\def\Mon{\textnormal{Mon}}

\def\zetaz{z}
\def\bran{\rm br}
\def\Vg{\text{Ver}(\gamma)}

\def\Mg{\text{Mon}(\gamma)}

\def\uwt{^{\textnormal wt}}

\title[Local Structures and Desingularization of $\MPd$]
{Genus-One Stable Maps, Local Equations, and
Vakil-Zinger's desingularization}

\author{Yi Hu}
\address{Department of Mathematics, University of Arizona, USA.}
\email{yhu@math.arizona.edu}
\author{Jun Li}
\address{Department of Mathematics, Stanford University, USA.}
\email{jli@math.stanford.edu}

\begin{document}
\maketitle

\begin{abstract}
We describe
an algebro-geometric approach to Vakil-Zinger's desingularization
of the  main component of the moduli of genus one stable maps to
$\Pn$ \cite{VZ0, VZ}. The new approach provides complete local
structural results for this moduli space as well as for the
desingularization of the entire moduli space and should
fully extend to  higher genera.
\end{abstract}

\section{Introduction}

 Let $\MPd$ be the moduli of degree $d$ genus one stable maps \cite{K} to
$\Pn$ and let $\MPd_0\sub\MPd$ be the primary component that is
the closure of the open subset of all stable morphisms $[u,C]\in
\MPd$ with smooth domains. In \cite{VZ0,VZ}, Vakil-Zinger found a
canonical desingularzation $\widetilde{M}_1(\Pn,d)_0$ of $\MPd_0$
by performing ``virtual'' blowing-ups on $\MPd$. They also showed
that for $\tilde f_0: \tilde\cX\to\Pn$ and $\tilde\pi: \tilde\cX\to
\widetilde{M}_1(\Pn,d)_0$,  the pull-back universal family over $
\widetilde{M}_1(\Pn,d)_0$, the direct image sheaf
$$\tilde\pi_{\ast} \tilde f_0\sta\sO_{\Pn}(5)
$$
is locally free. Those desingularizations are useful for applying Atiyah-Bott
localization formula to the hyperplane relation proved by Li-Zinger in \cite{LZ}.
While the desingularization result in \cite{VZ} is algebro-geometric,
its proof is analytic in nature.\black

This paper provides an algebro-geometric approach to these
desingularization results.
It will be a part of an algebro-geometric approach of the
relation between ordinary and reduced
genus one GW-invariants of complete intersection in products of
projective spaces \cite{CL}. It will also serve as our first step to generalize
the structure results on moduli spaces of genus one stable maps to
higher genera.

Our proof consists of two stages. At first, we use the classical method of
studying special linear series on curves to give an algebro-geometric proof
of the local equations of $\MPd$, obtained by Zinger in \cite[\S 2.3]{Zinger}.
After that, we ``modular'' blow up $\MPd$ and prove that the resulting stack
has smooth irreducible components. The blowup construction used in this paper 
follows that in \cite{VZ}.

We now briefly outline our approach.
Since a stable map $[u,C]$ to a projective space $\Pn$ is given by
$$u=[u_0,\cdots,u_n]: C\lra\Pn,\quad u_i\in H^0(u\sta \sO(1)),
$$
its deformation is determined by the combined deformation of the
curve $C$ and the sections $\{u_i\}$. Since moduli spaces of curves is smooth, the singularity of
$\MPd$ is caused by the non-locally freeness of the direct
image sheaf $\pi_* f^* \sO_{\PP^n}(1)$ of the universal family
\begin{equation*}
 \pi: \cX\to \MPd \and f: \cX \lra \Pn
\end{equation*}
of $\MPd$.

To study the non-locally freeness of the direct image sheaf,
by assigning to each stable map $[u,C]$ the divisor $u_0^{-1}(0)\sub
C$, locally we can view $\MPd$ as a stack over the Artin stack $\fD_1$ of
pairs $(C,D)$ of genus one nodal curves $C$ and effective divisors
$D \subset C$. 
Over
each chart $\cV\sub \fD_1$, by picking an auxiliary section of the
universal curve $\pi: \cC\to\cV$ with $\cD\sub\cC$ the tautological divisor, we construct explicitly a complex
$\sR^\cdot=[\sO_\cV^{\oplus (d+1)}\mapright{\varphi}\sO_\cV]$ 
whose sheaf cohomology gives the
cohomology $R^\cdot \pi\lsta\sO_\cC(\cD))$.\black

We then apply the deformation theory of nodal curves to derive a
simple explicit form of the homomorphism $\varphi$ in $\sR^\cdot$.
Under a suitable trivialization,
$$\varphi=(0,\zeta_{1}, \cdots, \zeta_{d}),
$$
where each $\zeta_i$ is a suitable product of the pull back of
regular functions whose vanishing loci are irreducible components
of nodal curves in $\fD_1$. This description enables us to derive
explicit local equations of $\MPd$ (see  Theorem
\ref{thm:localEquTree}). They are analogous to the equations
described by Zinger \cite[\S 2.3]{Zinger}.


 To construct the desingularization of the moduli space,  we introduce the
Artin stack $\fM\uwt_1$ of pairs $(C,w)$ of genus one nodal curves
$C$ with non-negative weights $w\in H^2(C,\ZZ)$, meaning that $w
(\Sigma)\geq 0$ for all irreducible $\Sigma\sub C$. The stack
$\fM\uwt_1$ is smooth and contains closed substacks $\Theta_k$
whose general point is a pair $(C,w)$ such that $C$ consists of 
a smooth elliptic curve $C_e$ with $k$ rational curves attached
to $C_e$ and the restriction of $w$ to $C_e$ is zero, while its
restriction to the other components of $C$ in non-zero.
We then blow up $\fM\uwt_1$
successively along $\Theta_1$, $\Theta_2$, $\cdots$, etc., to
obtain $\widetilde{\fM}\uwt_1$. The desingularization of $\MPd$ is
$$\widetilde{M}_1(\Pn,d)=\MPd\times_{\fM\uwt_1}\widetilde{\fM}\uwt_1.
$$

After a detailed study of the lifting of the local equations of
$\MPd$ mentioned earlier, 
we prove that the
irreducible components of $\widetilde{M}_1(\Pn,d)$ are smooth and
intersect transversally. We also prove that for each irreducible
component $\widetilde{M}_1(\Pn,d)_\mu\sub \widetilde{M}_1(\Pn,d)$
with $(\tilde\pi_\mu,\tilde f_\mu)$ the pull-back universal family
on $\widetilde{M}_1(\Pn,d)_\mu$, the direct image sheaf
$$\tilde\pi_{\mu\ast} \tilde f\sta_\mu\sO_{\Pn}(r)
$$ 
is locally free. It is of rank $dr$ over the desingularization of the primary component and of
rank $dr+1$ elsewhere.  We comment that the results for
the primary component  were proved by Vakil-Zinger \cite{VZ}.
Weighted graph was also used to study stable maps to $\bP^2$ by
Pandharipande \cite{Pand}.

This paper is organized as follows. In \S 2, we outline our
approach, stating the main desingularization theorems \ref{thm:m1}
and \ref{thm:m2}, and the main local structure theorems
\ref{thm:equ1} and \ref{thm:equ3}. In \S 3, we introduce the
notion of weighted rooted trees of weighted nodal curves.
In the following section, we state and prove the main structural result
of the direct image sheaf $\pi\lsta f\sta\sO_{\Pn}(k)$. Finally,
in \S 5 we prove the theorems stated in \S 2.



We'd like to thank Zinger for numerous suggestions on improving the
presentation of the paper and for pointing out several oversights.
The second author also thanks him for valuable discussion
during their collaboration. The first author was partially supported by NSA;
the second author was partially supported by NSF DMS-0601002.

Throughout the paper, we fix an arbitrary algebraically closed base field $\kk$.  All
schemes in this paper are assumed to be noetherian over $\kk$.

\section{Canonical desingularization of $\MPd$ }

In this section, we state our main results.

\subsection{The Artin stack of weighted nodal curves}
\label{wnc}

\begin{defi}
Let $C$ be a curve. We set
$$H^2(C,\ZZ)^+ =\{ w \in H^2(C, \ZZ) \mid w (\Sigma) \ge 0 \; \hbox{ for}\ C\sub \Sigma\ \text{irreducible} \}.
$$
 A weighted nodal curve is a pair $(C,w)$ consisting of a connected nodal curve $C$ and a
weight assignment $w \in H^2(C,\ZZ)^+$.
\end{defi}

\begin{say} For any scheme $S$ and  flat family of nodal curves
$\cC \lra S$, a weight assignment of $\cC/S$ is a section $w$ of
the sheaf $R^2 \pi_* \ZZ_\cC^+$, where  $R^2 \pi_* \ZZ_\cC^+
\subset R^2 \pi_* \ZZ_\cC$ is the subsheaf consisting of all
sections $w$ whose restrictions $w(s)$ to any closed
point $s \in S$ lies in $H^2(\cC_s, \ZZ)^+$. Here $w(s)$ is the
image of $w$ under the pullback homomorphism $R^2 \pi_* \ZZ_\cC
\lra H^2(\cC_s, \ZZ)$.

A flat family of weighted nodal curves
over $S$ is a pair $(\cC/S, w)$ of a flat family of nodal curves
over $S$ together with a weight assignment over $\cC/S$. We say
that two families of weighted nodal curves $(\cC/S, w)$ and
$(\cC'/S', w')$ are equivalent if there is an isomorphism $h:
\cC/S \lra \cC'/S$ such that $w=h^* w'$.

A weighted nodal curve is called stable if every smooth
ghost (weight 0) rational curve $B \subset C$, $B$
contains at least three nodes of $C$. Fixing a genus $g > 0$, we form a
groupoid
$$\fM\uwt_g: \hbox{(Schemes)} \lra \hbox{(Sets)}.
$$
that sends any scheme $S$ to the set of all equivalence classes
of flat families of  stable weighted nodal curves of genus $g$
over $S$.
\end{say}

Mimicking the proof that the stack $\fM_g$ of nodal curves of
genus $g$ is an Artin stack, we obtain

\begin{prop}
The groupoid $\fM\uwt_g$ is a smooth Artin stack of dimension 3g-3.
The projection $\fM\uwt_g \lra \fM_g$ by forgetting the weight
assignment is \'etale.
\end{prop}

Note that $\fM\uwt_g$ is the disjoint union of infinitely many
smooth components, each of which is indexed by the total weight of
the weighted curves. Because of the stablity requirement, each
component is of finite type.

\subsection{Blowups of $\MPd$} $\\$

To describe the desingularization of $\MPd$, the notion of core curve is
pivotal.

\begin{defi} The core of a connected genus-one curve $C$ is the unique
smallest (by inclusion) subcurve of arithmetic genus one. The core
of a weighted curve $(C, w) \in \fM\uwt_1$ is called ghost if the
induced weight on the core is zero.
\end{defi}

\begin{say}\lab{theta}
The stack $\fM\uwt_1$ contains an open substack
$\mathring \Theta_0$ consisting
of weighted curves with non-ghost cores.
The complement
$\fM\uwt_1 \setminus \mathring \Theta_0$ admits a natural
partition according to the
number of rational trees attached to the ghost core curves:
$\mathring\Theta_k$ is the subset of pairs $(C,w)$ such that $C$ can be
obtained from the ghost core $C_e \subset C$ by attaching $k$ (connected) trees of
rational curves to the core $C_e$ at $k$ distinct smooth points of
$C_e$. Then $\fM\uwt_1=\coprod_{k\geq 0}\mathring\Theta_k$.
We let $\Theta_k$ be the closure of $\mathring\Theta_k$.
\end{say}

\begin{say}\lab{inductive-blowups}
We can successively blow up $\fM\uwt_1$ along the loci $\Theta_k$. The locus
$\Theta_1$ is a Cartier divisor; blowing up along $\Theta_1$ does
nothing. We start by blowing up $\fM\uwt_1$ along the locus
$\Theta_2$, which  is a smooth codimension 2 closed substack of
$\fM\uwt_1$; we denote the resulting stack by $\fM\uwt_{1,[2]}$, which
is smooth. Inductively, after obtaining $\fM\uwt_{1,[k-1]}$, we blow
it up along the proper transform $\Theta_{[k-1],k} \subset
\fM\uwt_{1,[k-1]}$ of the closed substack $\Theta_k \subset
\fM\uwt_1$.  Since $\Theta_{[k-1],k}$ is a smooth closed substack of
$\fM\uwt_{1,[k-1]}$ of codimension $k$, the new stack is smooth. We
continue this process for all $k=2,3,\cdots$. Since each connected
component of $\fMw$ is of finite type, the blowup process on this
component will terminate after finitely many steps. Therefore, the limit
stack, which is the blowup of $\fMw$ along the proper transforms
of $\Theta_k$ in $\fM_{1,[k-1]}$ for all $k\geq 2$, is a well-defined smooth Artin
stack; we denote this stack by $\wfMw$.
\end{say}

\begin{say}
To induce a blowup on $\MPd$, we form the fiber product
$$\wMPd = \MPd \times_{\fMw} \wfMw.$$
Here the morphism $\MPd \to \fMw$ is defined as follows. The
first Chern class $c_1(f^*\sO_{\Pn}(1))$ gives a weight assignment to the
domain curves of the universal family $f: \cX \to \Pn$ of $\MPd $,
making $\cX$ a family of weighted nodal elliptic
curves. This family then defines a tautological morphism
$$\MPd \lra \fMw
$$
that is a lift of the tautological morphism
$\MPd \to \fM_1$.  Note that since $c_1(f^*\sO_{\Pn}(1))$
has degree $d$, defining $\wMPd$ requires
blowing up $\fMw$ along the proper transforms of $\Theta_k$
from $k=2$ to $k=d$. That is,
$$\wMPd = \MPd \times_{\fMw} \fM\uwt_{1,[d]}.$$
\end{say}

We can now succinctly reformulate the end result of
the virtual blowup construction of \cite[\S 4.3]{VZ}.

\begin{theo} \lab{thm:m1}
$\wMPd$ is a DM-stack with normal crossing singularities.
\end{theo}

For a refined version of this theorem with explicit local equations, see
Theorem \ref{thm:inductionCase=d}.

\begin{say}
Desingularizations of the sheaves $\pi_*f^*\sO(k)$
over $\wMPd$ are essential ingredients in computing
genus one GW-invariants of complete intersections  \cite{LZ, BCOV}.
The blowup $\wMPd$ contains a primary irreducible component
whose generic points are stable maps with non-ghost elliptic core.
We denote this component by $\wMPd_0$.
The other irreducible components
are indexed by the set of all partitions of $d$. For $\mu$ either $0$ or a partition of
$d$, let
$$\tilde\pi_\mu: \tilde\cX_\mu\lra\wMPd_\mu\and \tilde f_\mu:\tilde\cX_\mu\lra\Pn
$$
be the pull back of the universal map over $\MPd$.
\end{say}

The following theorem is due to Vakil-Zinger \cite{VZ}.

\begin{theo}\lab{thm:m2}
For every $k\geq 0$, the direct image sheaf $\tilde{\pi}_{\mu*}
\tilde{f}_\mu \sta\sO_{\Pn}(k)$ is a locally free sheaf over
$\wMPd_\mu$ {with $\mu$ either $0$ or a partition of $d$. It
is of rank $kd$ when $\mu=0$ and of rank $kd+1$ otherwise.}
\end{theo}


We will treat the two theorems in the reverse order
from what was done in \cite{VZ}. Specifically, we will first prove a
structure result for the direct image sheaf $\pi\lsta f\sta
\sO_{\Pn}(k)$ for all positive integers $k$. We
then derive local defining equations for $\MPd$ and
$\wMPd$, and obtain theorems \ref{thm:m1} and \ref{thm:m2} as corollaries.

\subsection{Local defining equations for $\MPd$}\lab{sec:eq1}

\begin{say}\lab{(C,D)}
For later use, we form the Artin stack of stable pairs $(C,D)$ of
(connected) nodal elliptic curves $C$ with effective divisors $D
\subset C$ supported on the smooth loci of $C$. Here $(C,D)$ is stable if
any smooth rational curve in $C$ disjoint from $D$ contains at least three
nodes of $C$.
We denote this
stack by $\fD_1$. It maps to $\fMw$ by
sending $(C,D)$ to $(C, c_1(D))$. The morphism $\fD_1\to\fM\uwt_1$ is
smooth and has connected fibers. In particular, the connected components
of $\fD_1$ are indexed by the degree of the
effective divisors.
\end{say}

\begin{say}\lab{simplDivisor} For any closed point $[u,C] \in \MPd$, we let
$$u=[u_0, \cdots, u_n]: C \lra \Pn, \quad u_i \in \Gamma (C, u^* \sO_{\Pn}(1)),
$$
be the associated stable morphism. By a change of homogeneous
coordinates on $\Pn$, we can assume that $u^{-1}_0(0) \subset C$
is a smooth simple divisor $D$ of degree $d$. Once $D  \subset
C$ is fixed, we can choose $u_0$ to be the constant section $1
\in \Gamma (C,  \sO_{C}) \subset \Gamma (C,
\sO_{C}(D))$. Under this convention, the remaining sections $u_1,
\cdots, u_n$ are uniquely determined by the morphism $u: C \to
\Pn$. Consequently, deforming $u: C \to \Pn$ is
equivalent to deforming the pair $(C, D)$ and the sections
$u_1, \cdots, u_n$.
\end{say}

\begin{say} We next choose a small
open neighborhood $[u,C]\in U$ in  $\MPd$. Let
$$\pi: \cX \lra U \and f:  \cX \lra \Pn$$ be the
the universal family over $U$ with  $\cX_0 \cong C$  the fiber
over $[u,C]$.
Let $\cS=f^*[x_0 =0]$ {(where $[x_0,\cdots,x_n]$ are the
homogeneous coordinates of $\Pn$)}. By shrinking $U$ if necessary,
we can assume that $\cS$ is away from the singular points of the
fibers $\cX/U$.
This way, we obtain a morphism
$$U \lra \fD_1, \quad [u', C'] \in U \longmapsto (C',  C'
\cap\cS).$$
\end{say}

\begin{say}\lab{deform}
We now construct the deformation space of the data $(C,
D, u_1, \cdots, u_n)$. We let $\cV \to  \fD_1$ be a smooth
chart of $(C, D) \in \fD_1$ that contains the image of $U\to\fD_1$; let $(\cC, \cD)$ be the
tautological family over $\cV$ with $(\cC_0, \cD_0) =(C, D)$
for some point $0 \in \cV$, and let
$$\rho: \cC \lra \cV
$$
be the
projection. We set $\sL = \sO_{\cC}(\cD)$. Set theoretically, our
deformation space is the union $\bigcup_{v \in \cV}
H^0(\cC_v, \sL|_{\cC_v})^{\oplus n}$. The deformation is
singular at points where the core curve of $\cC_v$ is ghost
due to $H^1(\cC_v,\sL_v)\ne 0$.

The algebraic construction of the deformation space is via a standard trick.
\end{say}

\begin{say}\lab{A}
By shrinking $\cV$ if necessary, we can find a section $\cA
\subset \cC$ of $\cC/\cV$, away from $\cD$, such that it passes through smooth
parts of the core curves of all fibers of $\cC/\cV$. Because
$\sL(\cA)$ is effective and has positive degree on the core curve
of every fiber of $\cC/\cV$, we have \beq R^1 \rho_* \sL(\cA)=0
\and \hbox{$\rho_* \sL(\cA)$ is locally free}.
\eeq
We let $\cE_{\cV}$
be the total space of the vector bundle $\rho_* \sL(\cA)^{\oplus
n}$ and let $p: \cE_{\cV} \to \cV$ be the projection.
Then the tautological restriction homomorphism
$$\text{rest}: \rho_* \sL(\cA)^{\oplus n}\lra \rho_* (\sL(\cA)^{\oplus n}|_\cA)
$$
lifts to a section
\beq \lab{sec1} F \in \Gamma
(\cE_{\cV}, p^*\rho_* (\sL(\cA)^{\oplus n}|_{\cA})).
\eeq
\end{say}

\begin{theo}\lab{thm:equ1} Let $\cU= \cV \times_{\fD_1} U$.
Then there is a canonical open immersion $\cU \to \;${\rm
(}$F=0${\rm )} $\sub \cE_{\cV}$.
\end{theo}

These local equations are made explicit in Theorem
\ref{thm:localEquTree}.


\begin{say}
We let $\widetilde\fM\uwt_1\to\fM\uwt_1$ be the blowup described in \ref{inductive-blowups}.
We form the fiber product
$$\widetilde\fD_1 = \fD_1
\times_{\fM_1\uwt} \widetilde\fM_1\uwt ,\quad \widetilde\cV=
\cV \times_{\fD_1}\widetilde\fD_1\and \widetilde\cU=\widetilde\cV\times_{\fD_1} U.
$$
Let $\eta: \widetilde\cV\to\cV$ be the projection. Then
$$\cE_{\widetilde\cV}=\cE_\cV\times_\cV\widetilde\cV
$$
is the total space of the pull back bundle $\eta\sta \rho\lsta\sL(\cA)^{\oplus n}$.
The immersion $\cU\to\cE_\cV$ of Theorem \ref{thm:equ1} induces an
immersion $\widetilde\cU\to \cE_{\widetilde\cV}$.
\end{say}

\begin{theo} \lab{thm:equ3}
Let $\xi: \cE_{\widetilde\cV}\to\cE_\cV$ be the projection and let
$\widetilde F=\xi\sta(F)$ be the pull back section in $\xi\sta p\sta\rho\lsta(\sL(\cA)^{\oplus n}|_\cA)$.
After shrinking $\cV$ if necessary and fixing a suitable
trivialization
$\xi\sta p^* \rho_*(\sL(\cA)^{\oplus n}|_{\cA})\cong\sO_{\cE_{\widetilde\cV}}^{\oplus n}$,
we can find $(d+n)$ regular functions $w_1,\cdots,w_n, \xi_1, \cdots, \xi_d$ over
$\cE_{\widetilde\cV}$ such that
$$
\widetilde F =(w_1 \xi_1\cdots\xi_d, \cdots, w_n \xi_1\cdots\xi_d).
$$
Further, each $w_i$ and $\xi_j$ has smooth vanishing locus and
the vanishing locus of their product $(w_1 \cdots w_n\cdot \xi_1
\cdots \xi_d=0)$ has normal crossing singularities.
\end{theo}

Note that some $\xi_1, \cdots, \xi_d$ may be invertible.
For more explicit local equations for $\wMPd$; see Theorem
\ref{thm:inductionCase=d}.

\section{Combibatorics of the Dual Graphs of Nodal Curves}

In this section, we discuss the combibatorics of the dual graphs
of nodal curves and introduce terminally weighted trees for
weighted elliptic nodal curves. This combibatorics is not strictly
necessary for our presentation, but they
do make our exposition more intuitive and formulas more elegant.

\subsection{Terminally weighted rooted trees}

\begin{say}
Let $\gamma$ be a connected rooted tree with $o$ its
root\footnote{All trees are connected in this article.}. The root $o$
defines a unique partial ordering on the set of all vertices of
$\gamma$ according to the descendant relation so that the root is the
unique {\sl minimal} element; all non-root vertices are
descendants of the root $o$. We call a vertex terminal if it has no
descendants; in the combinatorial world, this is also called a
leaf. Equivalently, a non-root terminal vertex is a non-root
vertex that has exactly one edge connecting to it. Following this
convention, the root is terminal only when $\gamma$ consists of a
single vertex. In the following, given a rooted tree $\gamma$ we
denote by $\Vg$ the set of its vertices, by $\Vg\sta$ its
non-root vertices, and by $\Vg^t$ its terminal vertices. Note that
for any vertex $v$ in $\gamma$, there is a unique (directed) path
between $o$ and $v$; this is the maximal chain of vertices
$o=v_0 \prec v_1 \prec \cdots \prec v_r=v.
$
\end{say}

\begin{say}
Next we consider weighted rooted trees. A weighted rooted tree is a
pair $(\gamma, w)$ consisting of a rooted tree $\gamma$ together
with a function on its vertices $w\mh \text{Ver}(\gamma)\to \ZZ^{\geq 0}$,
called the weight function. A vertex is positive if its weight is positive;
a ghost vertex is a vertex with zero weight. The total weight of the tree is the sum
of all its weights.
\end{say}

In this paper, we will consider only {\it terminally weighted}
rooted trees.

\begin{defi}
A weighted rooted tree $(\gamma, w)$ is terminal if all terminal
vertices are positive and all positive vertices are terminal; it
is called stable (resp. semistable) if every ghost non-root vertex
has at least three (resp. two) edges attached to it.
\end{defi}

When the weight $w$  is understood, we will use $\gamma$ to denote the weighted rooted
tree $(\gamma,w)$ as well as its underlying rooted tree $\gamma$ with weights removed.

\vskip 4.5cm

\begin{picture}(3, 15)
\put(45,2){ \psfig{figure=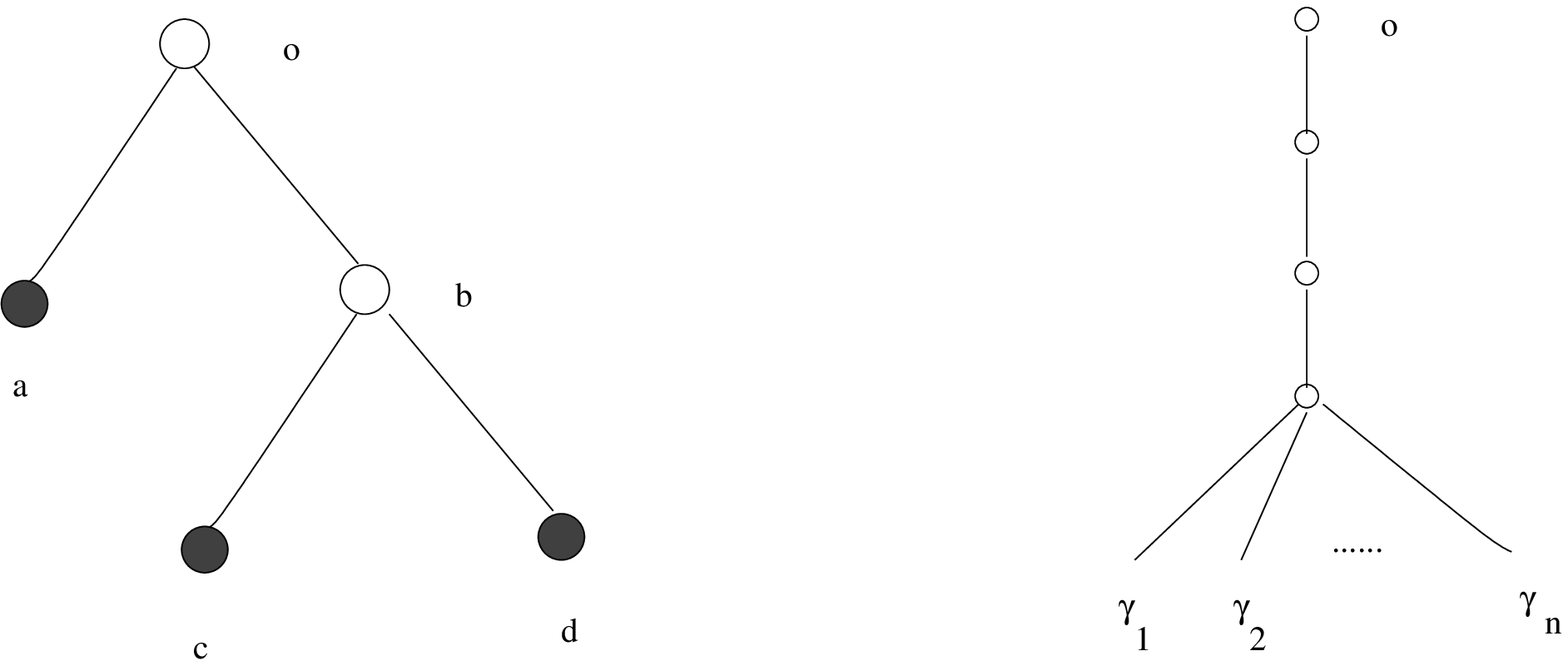, height = 4cm,width=9cm} }
\end{picture}

\vskip .5cm

\centerline{Figure 1.  Two trees}

\subsection{Operations on terminally weighted trees}

\begin{say}
For positive $d$, we let $\Lambda_d$ (resp. $\Lambda^{ss}_d$) be
the set of all stable (resp. semistable) terminally weighted
rooted trees of total weight $d$. The set $\Lambda_d$ is finite,
while $\Lambda_d^{ss}$ is infinite. The set $\Lambda_d^{ss}$
admits the following geometrical operations that will be useful
for our discussion.
\end{say}

\begin{say}
The first operation is {\it pruning} a tree. Given $\gamma\in
\Lambda_d^{ss}$, to prune $\gamma$ from a vertex $v$,  we simply
remove all descendants (i.e. those $u$ with $u\succ v$)
and the 
edges connecting the removed vertices. After pruning $\gamma$ from $v$, the new graph
has $v$ as its terminal vertex. 
If $\gamma$ is a weighted
tree, we define the weight of $v$ in the pruned tree to be the sum of
the original weight of $v$ and the weight of the removed vertices;
we keep the weights of the other vertices
unchanged. Note that the
resulting pruned tree is terminally weighted as well.

The second operation is {\it collapsing} a vertex. 
Collapsing a vertex $v$ in $\gamma$ is a two-step process: first
merge $v$ with its unique ascendent, removing the edge
between them and assigning the sum of their weights to the merged vertex,
and then prune the resulting tree along all
positively weighted non-terminal vertices, repeating the process as long
as possible.

The third operation is {\it specialization}: it is the inverse
operation of a collapsing.

The fourth operation is {\it advancing} a vertex. Let $v$ be a
vertex in $\gamma$ and let $\bar v$ be its direct ascendant. To advance $v$, replace every edge connecting $\bar{v}$
to a direct descendant $v_i$ other than $v$ by an edge connected $v_i$ to $v$
and then prune the resulting tree along all positively
weighted non-terminal vertices, repeating
the process as long as possible.

\vskip 9cm

\begin{picture}(3, 15)
\put(20,2){ \psfig{figure=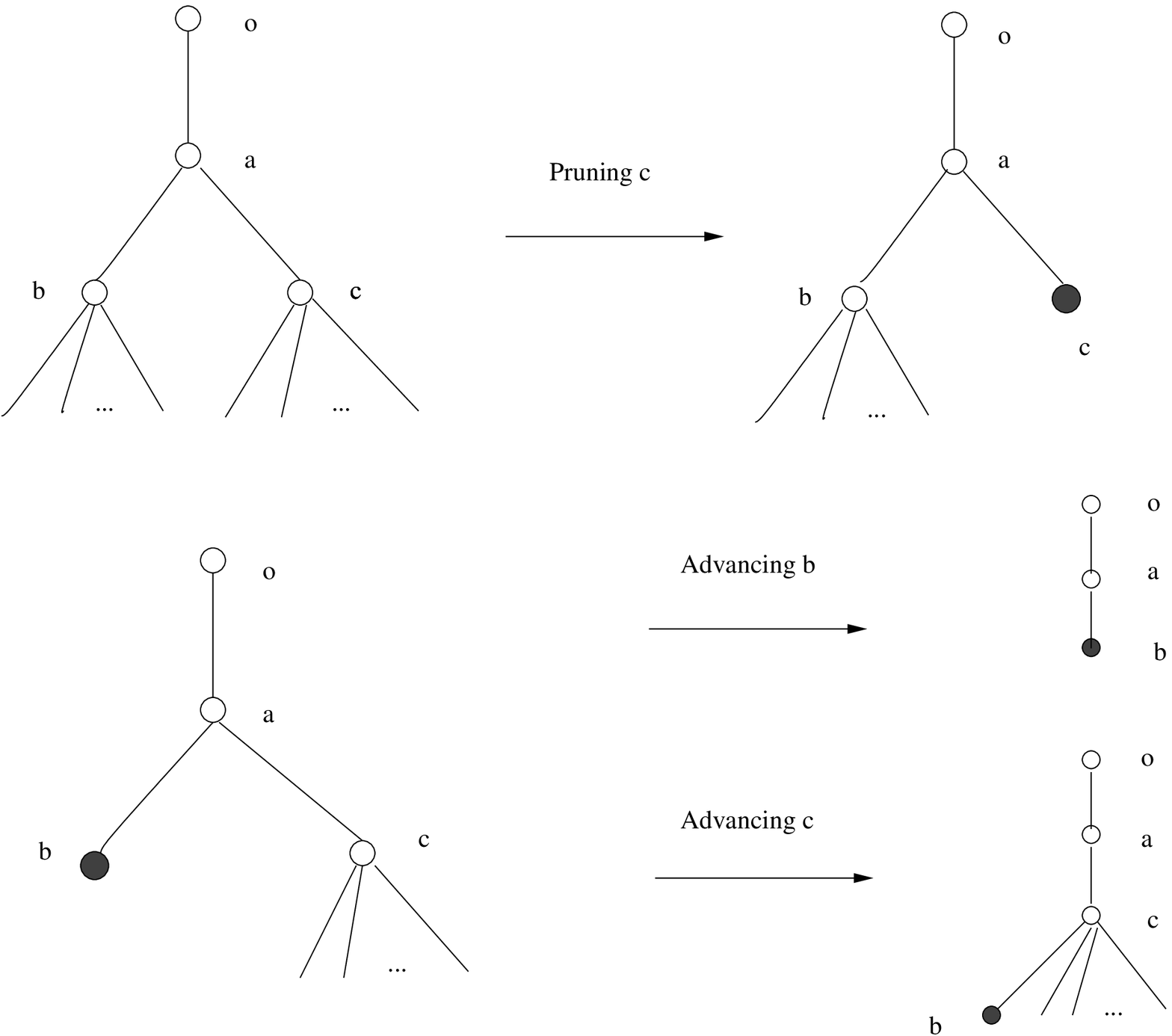, height = 8.5cm,width =
10cm} }
\end{picture}

\vskip .5cm

\centerline{Figure 2.  Operations on weighted trees}

\end{say}



\bigskip

\begin{say} Instead of drawing a picture, a weighted rooted tree can also be
described compactly as follows. Here is an example. Let
 $$\gamma=o[a(2),b[c(1),d(1)]]. 
 $$
This is a weighted rooted tree  whose root is $o$; the other vertices
are labelled by $a,b,c,d$. The vertices inside a square bracket
are the descendants of the vertex immediately proceeding the bracket; the
weights of the terminal vertices $a$, $c$, and $d$ are indicated in the
following parenthesis.
Collapsing $b$, we obtain $o[a(2),c(1),d(1)]$.
Advancing $b$, we get $o[b[a(2),c(1),d(1)]]$. Advancing
$a$, we obtain $o[a(4)]$.



When the weight function is irrelevant to the
discussion, we will drop any reference to weights. For example, the above tree
$\gamma$ would then be written as $\gamma=o[a,b[c,d]]$
\end{say}

\begin{say}
In this compact representation of a tree, pruning a
tree from a vertex $v$ means removing the bracket
immediately after $v$  and assigning it the total weight of the
vertices inside the bracket. Collapsing a ghost vertex $v$ is
removing the vertex $v$ as well as the closest brackets $``$[,]$"$
associated to it.
Advancing a ghost vertex $v$ is moving all the other
vertices located inside of the same square bracket as $v$ into the square
bracket following $v$.
\end{say}

\begin{say}
Observe that advancing can make a stable tree semistable. For
example, consider the stable tree $o[a[b[d,e], c]]$. If we advance
$b$, we obtain $o[a[b[c, d,e]]]$, which is  semistable but not
stable, since the vertex $a$ has only two edges attached to it.
\end{say}

\subsection{Monoidal transformations of weighted trees}\lab{MonTrees} $\\$

To keep track of the changes of the strata of $\MPd$ after blowups, we need
the notion of monoidal transformations of weighted trees. We begin with the following.

\begin{defi}
Let $\gamma$ be a semi-stable terminally weighted tree
with root $o$ and at least one non-root vertex.
The trunk of $\gamma$ is the maximal chain $o=v_0\prec\ldots\prec v_r$
of vertices in $\gamma$ such that each vertex $v_i$ with $i<r$
has exactly one immediate descendant.
\end{defi}

In this case, we abbreviate the trunk by
$\overline{ov_r}$; we call $v_r$ the branch vertex of $\gamma$ if it is not a {\sl terminal} vertex. Otherwise
we call $\gamma$ a path tree. When $v_r=o$,
we say the tree has no trunk.
Figure 1 shows two trees: the first one has no trunk; the second one has a trunk.

Let $\gamma$ be a tree with trunk $\overline{o\,v_r}$. Then
$\gamma$ can be obtained by attaching $\ell>1$ rooted trees, called branches,
 $\gamma_1',\cdots,\gamma_\ell'$, to the trunk so
that  the roots of $\gamma_i'$ are direct descendants of $v_r$. 
According to our
convention, $\gamma$ can be expressed as
$$\gamma=\overline{o\,v_r}[\gamma_1',\cdots,\gamma_\ell']:=o[v_1[\cdots[v_r[\gamma_1',\cdots,\gamma_\ell']]]].
$$
The tree $\gamma$ has no branches if and only if it is a path-tree.


\begin{defi}
Let $\textnormal{br}(\gamma)$ denote the number
of branches of $\gamma$.
We call $\gamma$ simple if all of the branches are stable.
\end{defi}



\begin{defi}
A monoidal transform of a simple terminally weighted tree $\ga$ is
a tree obtained by advancing one of the immediate descendants
of the branch vertex of~$\ga$, if $\ga$ has a branch vertex.
We denote the set of monoidal transforms of $\gamma$ by
$\Mg$.
\end{defi}

Note that every tree in $\Lambda_d$
is simple, and if $\ga$ is a simple tree, so is every monoidal transform of $\ga$.

\begin{lemm}\lab{mon1}
Let $\gamma$ be a simple terminally weighted tree and
$\tilde\gamma\in \Mg$. Then either
$\bran(\tilde\gamma)=0$, which is when $\tilde\gamma$ is a path-tree,
or $\bran(\tilde\gamma)\geq \bran(\gamma)+1$. The same
conclusion holds when $\tilde\gamma$ is a collapsing of $\gamma$ at a
direct descendant of the branch vertex.
\end{lemm}

\begin{proof} If $\tilde\ga$ is the result of advancing a direct descendant $v$
of the branch vertex $\bar{v}$ of $\ga$ and $v$ is not terminal,
then the direct descendants of $v$ in $\tilde\ga$ are the direct
descendants of $v$ in $\ga$ and the direct descendants of
$\bar{v}$ in $\ga$ other than~$v$. Furthermore, $v$ is the branch
vertex of $\tilde\ga$ in this case; thus,
$\textnormal{br}(\tilde\ga)\ge\textnormal{br}(\ga)+1$. On the
other hand, if $v$ is terminal in
$\ga$, then $\tilde\ga$ is the path from $o$ to $v$ in $\ga$. The
proof of the second statement is similar.
\end{proof}
\black

\begin{say}\lab{kmonTrees}  To index the strata of the various blowups of $\MPd$, we
introduce the following. We set $\Lambda_{d, [1]}=\Lambda_d$ and
define $\Lambda_{d, [k]}$ inductively for $k \ge 2$:
$$\Lambda_{d,[k]}=\{\gamma \in \Lambda_{d,[k-1]} | {\rm br} (\gamma) \ge k+1\} \cup
\{\gamma \in {\rm Mon} (\gamma') | \gamma' \in  \Lambda_{d,[k-1]},
{\rm br} (\gamma') =k \}.$$
\end{say}

\begin{lemm}\lab{mon2}
For any  $\gamma \in \Lambda_{d, [k]}$ with $1 \le k \le d$,
either $\bran (\gamma)=0$ or $k+1 \le \bran (\gamma) \le d$. In
particular,  $\Lambda_{d, [d]}$ consists of path trees only.
\end{lemm}

\begin{proof}  First,
the fact the total weight  of a terminally weighted tree $\gamma$
is $d$ implies that  $\bran (\gamma) \le d$. The assertion of the lemma holds
for $k=1$. 
On the other hand, Lemma \ref{mon1} implies that if the assertion holds for $k$, then
it also holds for $k+1$.
\end{proof}

\subsection{Terminally weighted trees of  weighted nodal curves}
\label{reducedGraph}

\begin{say}
Let $C$ be a connected nodal genus one curve. We associate to $C$ the dual $\ga_C'$, with vertices
corresponding to the irreducible components of $C$ and the edges to the nodes of~$C$.
Since the arithmetic genus of $C$ is $1$, either $\gamma_C'$ has a
unique vertex corresponding to the genus one irreducible component
of $C$ or $\gamma_C'$ contains a unique loop. In the first case,
we designate that vertex the {\sl root} of $\gamma_C\pri$; in the
latter case we shall contract the whole loop to a single vertex
and designate it the {\sl root} of the resulting tree. We will
call the resulting rooted tree the {\sl reduced dual tree} of $C$
and denoted it by $\gamma_C$.
\end{say}

\begin{say}
Next we consider a weighted nodal genus one curve $(C, w)$.
The weight $w$ induces
weights on the vertices of the dual graph $\gamma_C'$ of $C$. If $\gamma_C'$  contains a loop,
then upon contracting the loop, we assign the resulting root the total weight of the loop.
This way, we obtain a natural  weighted tree $(\gamma_C, w')$ associated to
the weighted curve $(C, w)$.

In general, the $\gamma_C$ with this weight function may not be
terminally weighted. If not, we can prune $\gamma_C$ along all
positively weighted non-terminal vertices to obtain a terminally weighted
rooted tree. We call the result the terminally weighted rooted tree of
$(C, w)$ and denoted it by $(\gamma_C, w)$.
In the case that $(C, w)$ is understood, we shall write it simply as
$\gamma$.

Following this construction, each terminal vertex $v$ of
$\gamma$ is a positively weighted vertex of $\gamma_C'$. We let
$C_v$ be the union of all irreducible components of $C$ whose associated vertices
$v'$ (in $\gamma_C'$) satisfy $v\preceq v'$. Then, $C_v$ is a tree of rational curves. 
Further, if $\gamma$ has zero weight root and
$v_1,\cdots,v_k$ are its terminal vertices, then $C-\cup_{i=1}^k C_{v_i}$
is the maximal weight zero connected subcurve of $C$ that contains the core
elliptic curve.
\end{say}


\subsection{A stratification of $\MPd$}

Any stable map $[u, C] \in \MPd$ naturally gives rise to
a weighted nodal genus one curve $(C, w)$, where the weight of an irreducible component
of $C$ is the degree of the map $u$ on that component. We will then call the
associated terminally weighted rooted tree of $(C, w)$ the terminally weighted rooted tree of
the stable map $[u]$. We denote it by $(\gamma_{[u]}, w)$. It is stable.

\begin{defi}
For any $\gamma \in  \Lambda_d$, we define $\MPd_\gamma$ be the
subset of all ${[u]}\in\MPd$ whose associated terminally weighted rooted
trees is $\gamma$.
\end{defi}

\begin{lemm}
\lab{prop:strataSmooth} Each $\MPd_\gamma$ is a smooth, locally
closed substack of $\MPd$; together they form a stratification of
$\MPd$. 
\end{lemm}

\begin{proof}
Suppose $\gamma\in { \Lambda_d}$ has $\ell$ terminal
vertices, indexed by $1,\cdots,\ell$ and of weights
$d_1,\cdots,d_\ell >0$ with $ \sum_{i=1}^\ell d_i = d$. Without the weight,
$\gamma$ is the {\sl reduced dual graph} of some  genus 1 nodal
curve. We denote by $M_{\gamma}$ the stratum in $M_{1,\ell}$
consisting of stable genus 1 curves whose {\sl reduced dual graph}
is $\gamma$ with terminal vertices replaced by the corresponding
marked points.
Then
$\MPd_\gamma$ is (up to equivalence by automorphisms)
$$\bigl\{([C_o,p_1,\cdot,p_\ell], [{u_i},C_i,q_i]_1^\ell)\in
M_{\gamma}\times\prod_{i=1}^\ell\overline M_{0,1}(\Pn,d_i) \mid
{u_1}(p_1)=\cdots={u_\ell}(p_\ell)\bigr\}.
$$
Since $\Pn$ has ample tangent bundle, $\overline M_{0,1}(\Pn,d_i)$ smooth and the evaluation
morphisms $u_i$ are submersions. Hence $\MPd_\gamma$ is smooth.
\end{proof}


\section{The Structure of the Direct Image Sheaf}

In this section, we state and prove  structure results of
the direct image sheaf $\pi\lsta f\sta\sO_{\Pn}(k)$.

\subsection{Terminology}

\begin{defi}\lab{sep} Let $C$ be a proper nodal curve with arithmetic genus $g >0$.
We call a node $q$ of  $C$ a separating node if $C-q$ is
disconnected.  Similarly, we call an irreducible component
$\Sigma\sub C$ a separating component if $C-\Sigma$ is
disconnected.
\end{defi}

Along the same line, we introduce

\begin{defi}
An inseparable curve is a connected curve with no separating
node; an inseparable
component of $C$ is an inseparable subcurve of $C$ that is not a proper subcurve
of another inseparable subcurve of $C$.
\end{defi}

\begin{say}\lab{nodes}
We say that a (separating) node $q$ separates $x$ and $y\in C$ if
$x$ and $y$ lie in different connected components of $C-q$; we say that the node $q$ lies between $x$
and $y$ in this case.
We let $N_{[x,y]}$ be the collection of all nodes that lie
between $x$ and $y$. This notion extends beyond nodes: for any
smooth point $t\in C$, we denote by $C_t$ the inseparable component of $C$ that
contains $t$; we say $t$ (or $C_t$) separates or lies between $x$
and $y\in C$ if $x$ and $y$ lie in different connected components
of $C-C_t$.
\end{say}

\begin{say}\lab{chooseab} For a nodal elliptic curve $C$ and
two distinct smooth points $a$ and $b$ on the
core of $C$, we have
\beq\lab{generalPoints1}
h^0( C, \sO_C(a-b))= 0\quad \hbox{and} \quad h^1( C, \sO_C(a-b))=
0;
\eeq
for any point $\delta$  of $C$ distinct from $a, b$,
we have \beq\lab{generalPoints2} h^0( C, \sO_C(\delta+ a-b))=
1\quad \hbox{and} \quad h^1( C, \sO_C(\delta+ a-b))= 0.\eeq
\end{say}

\begin{say}
Let $X$ be a scheme, $D$ a Cartier divisor of $X$, and $Z$ a closed
subscheme of $X$. We will write $\sO_Z(D)$ for the restriction
$\sO_X(D)|_Z$.
\end{say}

\subsection{The first reduction}
\label{1stRed}

\begin{say}
Our aim is to describe the structure of $\pi_*f^* \sO_{\Pn} (k)$; we let $m=dk$ in the rest of this section.
 When we investigate the
structure of $\MPd$ via $\pi_*f^* \sO_{\Pn} (1)$, we will specialize to $m=d$.
\end{say}

\begin{say}\lab{1stred}
Consider the substack $\fD_1^m$ of the Artin stack $\fD_1$ of
pairs $(C,D)$ of nodal elliptic curves with effective degree $d$
divisors $D \subset C$. Let $(C, D) \in \fD_1^m$ be a
point with the divisor $D$ simple and supported on the smooth
locus of $C$. We let $\cV \to \fD_1^m$ be a smooth chart
containing $(C, D)$. Again, let $(\cC, \cD)$ be the
tautological family over $\cV$ with ${(\cC_0, \cD_0)}
=(C, D)$ for some point $0 \in \cV$ and $\rho: \cC \to
\cV$ be the projection; set $\sL = \sO_{\cC}(\cD)$. As in \ref{A},
we choose a general section $\cA$ of $\cC/\cV$ and this time
around also an additional general section $\cB$ of $\cC/\cV$ such
that they are disjoint from $\cD$ and pass through the core of
every fiber of $\cC/\cV$. This is possible after shrinking $\cV$
if necessary. By the Mittag-Leffler exact sequence, the sheaf
$\rho_* \sL$ over $\cV$ is the kernel sheaf of \beq\lab{keyHom}
\psi: \rho_* \sL(\cA) \lra \rho_*\sO_{\cA}(\cA). \eeq
\end{say}

\begin{say} The complex of locally free sheaves of $\sO_\cV$-modules
$$ [R^\bullet]=[\rho_* \sL(\cA) \mapright{\psi}
\rho_*\sO_{\cA}(\cA)]
$$
has sheaf cohomology $[R^\bullet \rho_*\sL ]$. Further,
for any scheme $g: T \to \cV$ with the induced family
$$\rho_T: \cC_T = \cC \times_\cV T \lra T \and \cD_T =  \cD \times_\cV T,
$$
since $R^1\rho_* \sL(\cA)=R^1 \rho_*\sO_{\cA}(\cA)=0$, by
cohomology and base change, 
$$R^i \rho_{T*} \sO_{\cC_T} (\cD_T)\equiv h^i(g^*[R^\bullet]).
$$
\end{say}

\begin{say} To get hold of the
sheaf $\pi\lsta f\sta\sO_{\Pn}(k)$, we shall study the local
structure of the homomorphism \eqref{keyHom}.  As in
\ref{simplDivisor}-\ref{deform}, 
we only need to consider the case that
$D$ is a smooth simple divisor on $C$; we will assume that this
holds. We may also assume that $\cV$ is affine. After shrinking
$\cV$ and an \'etale base change if necessary, we may assume that
$\cD = \sum_{i=1}^m \cD_i$, where $\{\cD_i\}$ are disjoint sections
of the family $\cC/\cV$. To the sheaf $\sL=\sO_\cC(\cD)$,
the standard inclusion $\sO_{\cC}\sub \sO_{\cC}(\cD)$
provides us a section $1\in\Gamma(\rho\lsta\sL)$, called the
obvious section. To capture other sections, we consider the inclusion of
sheaves
$$ \sM_i=\sO_{\cC}(\cD_i+\cA-\cB)\mapright{\sub}
\sM=\sO_{\cC}(\cD+\cA-\cB)
$$
and the induced inclusions
$$\eta_i: \rho_*\sM_i \mapright{\sub}
\rho_*\sM.
$$
Both are locally free since $R^1\rho\lsta\sM_i$ and
$R^1\rho\lsta\sM=0$ by \eqref{generalPoints2}. By Riemann-Roch,
$\rho\lsta\sM_i$ is invertible and the rank of $\rho\lsta\sM$ is
$m$. We then let
$$\varphi: \rho\lsta\sM\lra
\rho\lsta\bl\sO_{\cA}(\cD+\cA-\cB)\br=\rho\lsta\bl\sO_{\cA}(\cA))
$$
and
$$\varphi_i: \rho\lsta\sM_i\lra
\rho\lsta\bl\sO_{\cA}(\cD_i+\cA-\cB)\br=\rho\lsta\bl\sO_{\cA}(\cA))
$$
be the evaluation homomorphisms. Obviously, $\varphi_i=\varphi
\circ \eta_i$. Since $\cV$ is assumed affine, the sheaf $\rho\lsta\bl\sO_{\cA}(\cA))
\cong \sO_\cV$.
\end{say}

\begin{lemm}
\lab{lemm:usefulFacts}
We have
\begin{enumerate}
\item $\rho\lsta\sL\cong\sO_\cV\oplus\rho\lsta
\sL(-\cB)$; \item $\rho\lsta\sL(-\cB)\cong
\ker\varphi$; \item $\oplus_{i=1}^m \eta_i: \bigoplus_{i=1}^m
\rho_*\sM_i\lra \rho_*\sM$ is an isomorphism, and $\oplus_{i=1}^m
\varphi_i={\varphi \circ \oplus_{i=1}^m\eta_i}.$
\end{enumerate}
\end{lemm}

Consequently, the sheaf $\rho\lsta\sL$ is a direct sum
of $\sO_\cV$ with the kernel of the homomorphism
\begin{equation}\lab{ker}
\Oplus_{i=1}^m \varphi_i: \rho\lsta\sM_i\lra\sO_\cV.
\end{equation}

\begin{proof}
Taking the direct image of the exact sequence
$$0\lra \sO_{\cC}(\cD-\cB)\lra \sO_{\cC}(\cD)\lra
\sO_{\cB}(\cD)\lra 0,
$$
we obtain the exact sequence
$$0\lra \rho\lsta\sO_{\cC}(\cD-\cB)\lra \rho\lsta\sO_{\cC}(\cD)
\mapright{\alpha} \rho\lsta\sO_{\cB}(\cD).
$$
Clearly, $\rho\lsta\sO_{\cB}(\cD)\cong \sO_\cV$. Also, $\alpha$ is
surjective because the obvious section $1\in
\Gamma(\rho\lsta\sO_{\cC}(\cD))$ maps surjectively onto
$\rho\lsta\sO_{\cB}(\cD)$. Finally, because $\cV$ is affine,
$\Ext^1 (\sO_\cV,\rho\lsta\sO_{\cC}(\cD-\cB))=0$. Therefore, the
sequence 
$$0\lra \rho\lsta\sO_{\cC}(\cD-\cB)\lra \rho\lsta\sO_{\cC}(\cD)
\mapright{\alpha} \rho\lsta\sO_{\cB}(\cD)\lra 0
$$
is exact and splits. This proves (1).

The second is obvious. We now obtain the third. Since both sheaves $\rho_*\sM_i$ and $\rho_*\sM$
are locally free, we only need to show that for any closed $z\in
\cV$, 
$$\oplus_{i=1}^m \rho_*\sM_i\otimes\kk(z)\to \rho_*\sM\otimes\kk(z)
$$ 
is an
isomorphism. Because $R^1\rho\lsta$ of $\sM_i$ and $\sM$ are zero,
by base change, this is equivalent to that the tautological
homomorphism
$$\bigoplus_{i=1}^m \eta_{i}(z):\bigoplus_{i=1}^m H^0\bl
\cC_z,\sO_{\cC_z}(\cD_i+\cA-\cB)\br \mapright{} H^0\bl
\cC_z,\sO_{\cC_z}(\cD+\cA-\cB)\br
$$
is an isomorphism. Because both sides are of equal
dimensions, it suffices to show that
it is injective. For this, we look at the composite
of $\oplus \eta_i(z)$ with
$$\phi_j(z):H^0\bl
\cC_z,\sO_{\cC_{{z}}}(\cD+\cA-\cB)\br \lra H^0\bl\cC_z\cap\cD_j,
\sO_{\cC_z\cap\cD_j}(\cD+\cA-\cB)\br.
$$
Obviously, $\phi_j(z)\circ\eta_i(z)=0$ for $i\ne j$ and is an
isomorphism for $j=i$. This shows that $\oplus_{i=1}^m\eta_i(z)$ is
injective. This proves the last claim of the lemma. 
\end{proof}

\begin{say}
The homomorphism
$$\varphi: \rho\lsta\sM\to
\rho\lsta\bl\sO_{\cA}(\cD+\cA-\cB)\br\cong\rho_* \sO_\cA (\cA)
$$
then is completely determined by the homomorphism \beq
\lab{homsum} \Oplus_i \varphi_i: \bigoplus_i \rho_*\sM_i \lra
\rho\lsta\bl\sO_{\cA}(\cD+\cA-\cB)\br\cong\rho_* \sO_\cA (\cA).
\eeq

The homomorphism
\beq\varphi_i: \rho_*\sM_i\lra \rho_* \sO_\cA (\cA)
\eeq
will be our focus in the next subsection.
\end{say}

\subsection{The homomorphism $\varphi_i$}\lab{singleHom}

\begin{say}\lab{zeta-q}
Our strategy is to find an explicit expression for $\varphi_i$ so
that its vanishing locus has precise geometric meaning. {For
this,} we need {some}  regular functions associated to the
smoothing of nodes. By the deformation theory of nodal curves,
for each separating node $q\in C$ there is a regular
function $\zeta_q\in\Gamma(\sO_\cV)$ so that
$\Sigma_q=\{\zeta_q=0\}$ is the locus where the node $q$ is not
smoothed; the divisor $\Sigma_q$  is an irreducible smooth Cartier
divisor.

For any $1 \le i \le m$,
we introduce
$$\delta_i = \cD_i \cap C \and a = \cA \cap C.$$ We then collect all the
nodes $q$ that lie between $\delta_i$ and $a$ (cf. \ref{nodes}) and
form the product of their associated functions $\zeta_q$:
\beq\lab{product} \zeta_{[\delta, a]}=\prod_{q\in
N_{[\delta_i, a]}}\zeta_q.
 \eeq
In case $N_{[\delta_i, a]}=\emptyset$, we set $\zeta_{[\delta_i, a]}=1$.
\end{say}

\begin{prop}\lab{keyProp} There are trivializations $\rho_*\sM_i\cong\sO_\cV$ and $\rho_*
\sO_\cA(\cA) \cong \sO_\cV$ such that the homomorphism $\varphi_i$
is given by \beq\lab{explicitHom} \varphi_i= \zeta_{[\delta_i,
a]}: \rho_*\sM_i\lra \rho_* \sO_\cA(\cA). \eeq
\end{prop}
\proof Since $\cV$ is affine, we  fix a
trivialization $\rho_*\sM_i \cong\sO_\cV$, and keep the trivialization
$\rho\lsta\sO_\cA(\cA)\cong\sO_\cV$ mentioned before. This way,
$$\varphi_i\in \Gamma((\rho_*\sM_i)\dual\otimes\rho\lsta\sO_{\cA}(\cA))\cong\Gamma(\sO_\cV).
$$
Then the proposition is
equivalent to
that as divisors,
\beq\lab{vanishingLocus} \varphi_i \upmo(0)=\sum_{q \in
N_{[\delta_i, a]}}\zeta\upmo_q(0)= \sum_{q \in
N_{[\delta_i, a]}} \Sigma_q. \eeq

We next let
$\eta: R\to \cV$ be either a point or a smooth affine curve,
we let $\pi_R: \cC_R\to R$ be $\cC_R=\cC\times_{\cV}R$ over $R$, and let
$$\cD_{R,i}= \cD_i \times_\cV R,\quad {\cA_R}= \cA \times_\cV R
\and {\cB_R}= \cB \times_\cV R
$$ 
be the corresponding pull back divisors.
Since $R^1\rho\lsta\sM_i=0$, by cohomology and base change,  the natural homomorphism
$$\eta\sta \rho_*\sM_i=\eta\sta\rho\lsta\sO_{\cC}(\cD_i-\cB +\cA)\mapright{}
\pi_{R\ast}\sO_{\cC_R}(\cD_{R,i}-\cB_R+\cA_R)\cong
\sO_R
$$
is an isomorphism. Finally, let $\varphi_{R,i}$ be the composite
$$\varphi_{R,i}: \eta\sta \rho_*\sM_i\mapright{\cong} \pi_{R\ast}\sO_{\cC_R}(\cD_{R,i}-\cB_R+\cA_R )
\lra \pi_{R\ast}\sO_{\cA_R}(\cA_{R})\cong \sO_R.
$$
Then, if $R$ is a smooth curve not contained in $\varphi_i\upmo(0)$,  as divisors,
$$\eta\upmo(\varphi_i\upmo(0))=\varphi_{R,i}\upmo(0)\sub R.
$$

We now prove the claim. First let $R$ be a smooth point
away from $\cup_{q \in N_{[\delta_i, a]}} \Sigma_q$. Then  $\cA_R$,
 $\cB_R$, and $\cD_{R,i}$ lie in the same inseparable component
of $\cC_R$. Therefore $\pi_{R\ast}(\sO_{\cC_R}(\cD_{R,i}-\cB_R))=0$
and $\varphi_{R,i}\ne 0$ because
$$\ker\{\varphi_{R,i}\}=\pi_{R\ast}(\sO_{\cC_R}(\cD_{R,i}-\cB_R))=0.
$$

If $R$ is a smooth point in $\Sigma_q$ for some $q \in
N_{[\delta, a]}$, then $\cB_R$ and $\cD_{R,i}$ lie in
different inseparable components of $\cC_R$. This time,
$\pi_{R\ast}(\sO_{\cC_R}(\cD_{R,i}-\cB_R))\cong \sO_R$. Therefore,
for the same reason as above, $\varphi_{R,i} = 0$. This proves that $\varphi_i$
vanishes exactly along $\cup_{q \in N_{[\delta_i, a]}} \Sigma_q$.

It remains to show that $\varphi_i$ vanishes at first order only
along $\Sigma_q$, $q \in N_{[\delta_i, a]}$. 
To prove this we only need to study $\varphi_i$ near a
general point $p \in \Sigma_q$. We let $R\sub\cV$ be an affine curve passing through
$p$ and transversal to $\Sigma_q$ at $p=R\cap\Sigma_q$.
After shrinking $p\in R$ if necessary, the family $\cC_R\to R$ is the blowup
of a family of smooth elliptic curves $\bar\pi_R:E_R\to R$ at a point $q\in E_p=E_R\times_R p$.

We let $\xi: \cC_R\to E_R$ be the projection and let
$\cF_0\sub \cC_p$ be the rational component, which is also the exceptional divisor of $\cC_R$.
Let $A=\xi(\cA_R)$, $B=\xi(\cB_R)$,
and $D_i=\xi(\cD_{R,i})$ be the
image divisors in $E_R$.
Then
$$\xi\upmo(A)=\cA_R,\quad \xi\upmo(B)=\cB_R, \and \xi\upmo(D_i) =\cF_0+\cD_{R,i}.
$$
Further, since $\sO_{\cF_0}(\cD_{R,i}+\cF_0)\cong\cO_{\cF_0}$, the cokernel of the
inclusion
\beq\lab{cok}
\pi_{R\ast}\sO_{\cC_R}(\cD_{R,i}-\cB_R+\cA_R)\lra \pi_{R\ast}\sO_{\cC_R}(\cD_{R,i}+\cF_0-\cB_R+\cA_R)
\eeq
is $ \pi_{R\ast}\sO_{\cF_0}(\cD_{R,i}+\cF_0)$, which is isomorphic to $k(p)$.
Therefore, since $\varphi_{R,i}|_p=0$,
$\varphi_{R,i}$ factors through a homomorphism  $\phi$ as shown in the commutative diagram
$$\begin{CD} \pi_{R\ast}\sO_{\cC_R}(\cD_{R,i}+\cF_0-\cB_R+\cA_R)
@>{\phi}>> \pi_{R\ast}\sO_A(A)\cong \sO_R\\
@| @|\\
\bar\pi_{R\ast}\sO_{E_R}(D_i-B+A)@>>> \bar\pi_{R\ast}\sO_A(A)\cong \sO_R.
\end{CD}
$$
Since the lower horizontal arrow is an isomorphism, $\phi$ is an isomorphism.
Combined with that the cokernel of \eqref{cok} is $k(p)$, 
this proves that $\varphi_{R,i}$ has precisely order one vanishing at $p\in R$.
\endproof

For the convenience of reference, we record an immediate consequence
of Proposition \ref{keyProp}.

\begin{coro}\lab{maincoro}
There are trivializations $\rho_*\sM_i\cong\sO_\cV$ and $ \rho_* \sO_\cA(\cA) \cong
\sO_\cV$ such that the homomorphism $\varphi$ is given by
\beq\lab{mainHom} \Oplus_{i=1}^m \varphi_i:
 \bigoplus_{i=1}^m \rho_*\sM_i\lra \rho_*
\sO_\cA(\cA), \quad \varphi_i= \zeta_{[\delta_i, a]} . \eeq
\end{coro}

\begin{say}\lab{node-vertex}
The homomorphism \eqref{mainHom} can be further simplified. Recall that $\cV$ is a neighborhood
of $(C, D) \in \fD_1^m \sub \fD_1$. The pair
$(C, D)$ induces a weighted curve $(C, w)$ with $w=c_1(D)$. We let
$\gamma$ be the terminally weighted tree of $(C, w)$ with terminal vertices 
$$\Vg^t=\{1, \cdots, \ell\}.
$$
According to our convention, each non-root
vertex $v\in\Vg$ corresponds to a connected subcurve $C_v\sub C$ (c.f. \S\ref{reducedGraph});
on the subcurve $C_v$ there is a unique separating node $q$ of $C$  that
separates $C_v$ and the remainder part $C-C_v$.
We call this node $q$ the node
associated to $v$. With such node identified, for each vertex $v$ we define
$$\zeta_v=\zeta_q\in \Gamma (\sO_\cV),
$$
where $q$ the associated node of $v$. \black
For any
terminal vertex $i \in \Vg^t$, we let
$$\zeta_{[i,o]} =  \prod_{i \succeq v\succ o}\zeta_v.$$
\end{say}

\begin{theo} \lab{thm:key}
The direct image sheaf $\rho_* \sL $ is a direct sum of
$\sO_\cV^{\oplus (m-\ell+1)}$ with the kernel sheaf of the
homomorphism \beq \lab{simHom}
\Oplus_{i=1}^\ell\varphi_i:\sO_\cV^{\oplus \ell}\lra \sO_\cV,
\quad \varphi_i=\zeta_{[i,o]}. \eeq
\end{theo}

 \proof We express $D$ as $\sum_{j=1}^m \delta_j$, and continue to denote by
$v_1, \cdots, v_\ell$ the terminal vertices of
$\gamma$.  By our construction of the terminally weighted tree
$\gamma$ of the weighted curve $(C, w)$ of the pair $(C,D)$
(cf. \S \ref{reducedGraph}), each $v_i$ is associates to a connected
tree $C_{v_i}$ of rational curves; 
there is a unique irreducible component $D_{v_i}$ of $C_{v_i}$ 
closest to the core of $C$. Again, by the
construction of $\gamma$, at least one of
$\{\delta_j\}_{j=1}^m$ lies on $D_{v_i}$; we pick one and index it
by $\delta_i$. Further, for every $\delta_j, 1 \le j \le m$,  there
is $1 \le i\leq \ell$, such that $\delta_i$ is between the
point $a$ (of the core curve of $C$) and $\delta_j$ ($\delta_i$ and $\delta_j$ can
be on the same irreducible component). This shows
that $\zeta_{[\delta_i, a]} | \zeta_{[\delta_j, a]}$; thus,
every $\zeta_{[\delta_j, a]}$ is divisible by one of
$\zeta_{[\delta_1, a]},\cdots,\zeta_{[\delta_\ell, a]}$.

Thus in the expression \eqref{mainHom}, we can choose a new basis
for $\bigoplus_{j} \rho_*\sM_j$ so that with respect to the new
trivialization
$$\bigoplus_{j=1}^m \rho\lsta\sM_j\cong 
\sO_\cV^{\oplus m},
$$
the homomorphism $\varphi= \Oplus_{i=1}^m \varphi_i$ has the form
$$\varphi= \Oplus_{i=1}^\ell \varphi_i \oplus 0:  \sO_\cV^{\oplus m}
\lra \sO_\cV.
$$
Together with Lemma \ref{lemm:usefulFacts} and Corollary \ref{maincoro}, this proves the
theorem.
\endproof


\section{{Local Equations of $\MPd$ and its
Desingularization}}
\label{localStructures}

In this section, we prove the theorems stated in \S 2. In the meantime,
we describe local defining equations for $\MPd$ in terms of weighted
trees.

\subsection{Proof of Theorem \ref{thm:equ1}}
\lab{localEqua1}

\begin{say}
Recall that in Theorem \ref{thm:equ1}, for any $[u,C]\in\MPd$ we first pick a small open subset
$[u,C]\in U\sub \MPd$ and a homogeneous coordinate $[x_0,\cdots,x_n]$ of $\Pn$ so that
the pull back divisor $\cS=f\upmo[x_0=0]$ is a family of simple divisors on the domain family $\cX$ of the
universal family $f: \cX\to\Pn$ of $U$; the family $\cX$ coupled with the divisor $f\upmo[x_0=0]$
defines a tautological morphism $U\to \fD_1^d \sub \fD_1$. We next pick a smooth chart $\cV\to\fD_1$ so that its
image contains the image of $U\to\fD_1$. Let $\cU=\cV\times_{\fD_1}U$.
\end{say}

\medskip\noindent
{\bf Theorem \ref{thm:equ1}} {\it 
There is a canonical open immersion $\cU\to (F=0)\sub
\cE_{\cV}$.}

\begin{proof}
We continue with the notation introduced in \S 2. For instance,
 $(\cC, \cD)$ is the tautological family over
$\cV$. We set
$$\cX'=\cX \times_U  \cU, \quad \cD'= \cS \times_U \cU.$$
By the universality of $\fD_1$,
$$\cX' = \cC \times_\cV \cU, \quad \cD'= \cD \times_\cV \cU.$$

We use $\alpha$ and $\tilde\alpha$ to denote the induced
horizontal maps in the square \beq\lab{square}
\begin{CD}
\cX' @>{\tilde\alpha}>> \cC \\
@V{\pi'}VV @V{\rho}VV  \\
\cU @>{\alpha}>> \cV.
\end{CD} \eeq
Likewise, we denote $\sL'=\sO_{\cX'} (\cD')$ and  $\sL=\sO_{\cC}
(\cD)$. Then $\sL'=\tilde\alpha^* \sL$.

We now construct the promised open immersion
\beq\lab{mu}
\mu: \cU\lra \cE_\cV.
\eeq
We let $f'$ be the composition of the projection $\cX' \to \cX$ with
$f: \cX\to \PP^n$; let $s_i = f'^* (x_i)$, $0\leq i\leq n$, be the pull back sections
in $\Gamma (\pi'_*\sL')$.
According to our convention,
$s_0= f'^* (x_0)$ is the section $1$ induced by the inclusion
$\sO_{\cX'}\sub \sL'=\sO_{\cX'}(\cD')$. This way, all other $s_i$, $i\geq 1$, are
canonically defined.

As mentioned in \ref{A}, we have a section $\cA$ of $\cC/\cV$.
Its pull back section in $\cX'$ is $\cA'= \cA \times_\cV \cU$. Because $R^1 \rho_*\sL(\cA)=0$,
by the cohomology and base change
theorem, we have canonical isomorphism
\beq\lab{Can-iso}
 \alpha^*\rho_*\sL(\cA) \mapright{\cong}\pi'_* \sL'(\cA').
\eeq
We let
$$\iota: \pi'_* \sL'\lra \pi'_* \sL'(\cA') \cong \alpha^*\rho_*\sL(\cA)
$$
be the inclusion. Then $\iota(s_i)$ is a section of $\alpha^*\rho_*\sL(\cA)$.
On the other hand, since $\cE_\cV$ is the vector bundle $\rho\lsta\sL(\cA)^{\oplus n}$, defining
a $\cV$-morphism $\mu:\cU\to\cE_\cV$ is equivalent to giving a section of
the pull back sheaf $\alpha\sta\cE_\cV=\alpha\sta \rho\lsta\sL(\cA)^{\oplus n}$.
We define the morphism $\mu$ in \eqref{mu} to be the one induced by the homomorphism
\beq\lab{iotas}
\iota(s)=(\iota(s_1),\cdots,\iota(s_n)):\sO_\cU\lra \alpha\sta\rho\lsta\sL(\cA)^{\oplus n}.
\eeq

To complete the proof, we need to show that $\mu$ factors through $(F=0)\sub \cE_\cV$
and the factored morphism $\mu':\cU\to (F=0)$ is an open immersion.

We first check that $\mu$ factors. By definition, $\mu$ factors if the pull back $\mu\sta(F)\equiv 0$.
Let $p:\cE_\cV\to\cV$ be the projection. By definition, $F$ is the composite 
$$F: \sO_{\cE_\cV}\mapright{\bone}p\sta\rho\lsta\sL(\cA)^{\oplus n}\mapright{r} p\sta\rho\lsta(\sL(\cA)^{\oplus n}|_{\cA}),
$$
where $\bone$ is the
tautological section and $r$ is the restriction homomorphism.
Therefore, $\mu\sta(F)$ is the composite
$$\sO_\cU\mapright{\mu\sta(\bone)} \alpha\sta \rho\lsta\sL(\cA)^{\oplus n}\mapright{\mu\sta(r)}
\alpha\sta \rho\lsta(\sL(\cA)^{\oplus n}|_{\cA}).
$$
Since $\alpha=p \circ \mu$, $\mu\sta(\bone)$ is the composite
$$\sO_\cU\mapright{(s_\cdot)} \alpha\sta \rho\lsta\sL^{\oplus n}\mapright{\iota}
\alpha\sta \rho\lsta\sL(\cA)^{\oplus n}.
$$
Therefore $\mu\sta(F)$ is the composite
$$\sO_\cU\mapright{(s_\cdot)} \alpha\sta \rho\lsta\sL^{\oplus n}\mapright{\iota}
\alpha\sta \rho\lsta\sL(\cA)^{\oplus n}\mapright{\mu\sta(r)}
\mu\sta p\sta \rho\lsta(\sL(\cA)^{\oplus n}|_{\cA}).
$$
Since $\mu\sta(r)\circ \iota=0$, we get $\mu\sta(F)=0$. This proves that $\mu$ factors through
$$\mu': \cU\lra (F=0)\sub\cE_\cV.
$$

We next prove that $\mu'$ is an open immersion. We let $Z=(F=0)\sub\cE_\cV$
and let
$\tau: Z\lra\cE_\cV $
be the tautological immersion.
 Because $\cE_\cV$ is the total space of the vector bundle
$\rho\lsta\sL(\cA)^{\oplus n}$ on $\cV$, the morphism $\tau$ is equivalent to
giving a section (homomorphism)
$$s_\tau: \sO_Z\lra \tau\sta p\sta \rho\lsta\sL(\cA)^{\oplus n}.
$$
At the same time, $\tau\sta(F)=0$ is equivalent to the vanishing of the composite of the
homomorphisms:
\beq\lab{122}
\sO_Z\mapright{s_\tau} \tau\sta p\sta\rho\lsta\sL(\cA)^{\oplus n}
\mapright{\tau\sta(r)}  \tau\sta p\sta \rho\lsta(\sL(\cA)^{\oplus n}|_\cA).
\eeq
We remark that because $\alpha=p\circ\tau\circ \mu'$, by the universality property of morphisms to $\cE_\cV$,
\beq\lab{s-equal}
\mu'^*(s_\tau)=\iota(s): \sO_\cU\lra \mu'^*\tau\sta p\sta\rho\lsta\sL(\cA)^{\oplus n}
\equiv \alpha\sta\rho\lsta\sL(\cA)^{\oplus n}.
\eeq

To continue, we will show that the vanishing \eqref{122} provides us a family of stable morphisms
parameterized by an open subset of $Z$ that contains $\mu'(\cU)$.
We let
$$\cC_Z=\cC\times_\cV Z,\quad \cD_Z=\cD\times_\cV Z,\quad \cA_Z=\cA\times_\cV Z,
\and \sL''=\sO_{\cC_Z}(\cD_Z).
$$
Because $R^1 \rho\lsta\sL(\cA)=0$,
by the cohomology and base change theorem, we have the canonical identity
$$\Gamma(Z, \tau^*p\sta\rho_* \sL(\cA)^{\oplus n})=\Gamma(\cC_Z, \sL''(\cA_Z)^{\oplus n}).
$$
This identity transforms \eqref{122}  into
\beq\lab{222}
\sO_{\cC_Z}\mapright{s_\tau''} \sL''(\cA_Z)^{\oplus n}
\mapright{r''}  \sL''(\cA_Z)^{\oplus n}|_{\cA_Z}.
\eeq

Because $r''\circ s_\tau''$ is zero and the kernel of the second arrow
is $\sL''^{\oplus n}$,
$s_\tau''$ factors through a unique homomorphism
\beq\lab{spp}
s''=(s''_{1},\cdots,s_{n}''): \sO_{\cC_Z}\lra \sL''^{\oplus n}.
\eeq
Let $s''_0$ be the section $1$ of $\sO_{\cC_Z}\sub\sL''=\sO_{\cC_Z}(\cD_Z)$.
The $(n+1)$ sections $s_0'',\cdots,s_n''$ considered as sections of $\sL''$ on $\cC_Z$
define a morphism
\beq\lab{smorpp}[s_0'',\cdots,s_n'']: \cC_Z \setminus \{s_0''=\cdots= s''_n=0\}\lra \Pn.
\eeq

To analyze the domain of this morphism, we notice that due to
\eqref{s-equal}, 
the morphism
\beq\lab{equal2}
[s_0,\cdots,s_n]=[s_0'',\cdots,s_n'']\circ \tilde\mu': \cX'=\cX\times_U\cU=\cC\times_\cV U \lra \Pn,
\eeq
where $\tilde\mu'$ is the lift of $\mu'$ to $\cX'$.
Therefore
$$\cC_Z\times_Z{\mu'(\cU)}\sub \cC_Z\setminus \{s_0''=\cdots= s''_n=0\}.
$$
Because $\cC_Z\to  Z$ is proper, there is an open $W\sub Z$ containing $\mu'(\cU)$
so that
$$\cC_W=\cC_Z\times_Z W\sub\cC_Z\setminus \{s_0''=\cdots= s''_n=0\}.
$$
We let
$$f_W: \cC_W\lra \Pn
$$
be the restriction of \eqref{smorpp} to $\cC_W$. Finally, because restricting to $\mu'(\cU)$ this
morphism is a family of stable morphisms, possibly after shrinking $W\supset \mu'(\cU)$ if necessary,
$f_W$ is a family of stable morphisms.

We let
$$\eta: W\lra \MPd
$$
be the tautological morphism induced by the family $f_W$. Because of the identity \eqref{equal2},
the composite of $\mu': \cU\to W$ with $\eta: W\to\MPd$ is identical to the projection
$\cU=\cV\times_{\fD_1} U\to U\sub\MPd$. Therefore, if we let
$$W_0=\eta\upmo(U),
$$
$W_0\sub Z$ is open and $\mu'$ factor through
$$\mu'': \cU\lra W_0.
$$

We claim that, with $W_0\sub Z$ endowed with the open subscheme structure of $Z$, the morphism
$\mu''$ is an isomorphism.
To prove this, we will construct the inverse of $\mu''$. Let $\eta'': W_0\to U$ be the morphism induced by $\eta$. Because
the composite $\eta'': W_0\to U$ with the tautological $U\to\fD_1$ is identical to the composite of
$p\circ\tau: W_0\to \cE_\cV\to\cV$ with $\cV\to\fD_1$, the pair $(\eta'', p\circ\tau)$ lifts to a morphism
$$\zeta'': W_0\lra \cU=\cV\times_{\fD_1} U.
$$
Because of the identity \eqref{equal2}, the composite $\eta\circ\mu''$ is identical to the projection $\cU\to U$.
This implies that $\zeta''\circ\mu''=\text{id}_{\cU}$. On the other hand, $\mu\circ\zeta'':W_0\to\cE_\cV$
is exactly the inclusion $W_0\to\cE_\cV$, again due to the identity \eqref{equal2}, therefore
$\mu''\circ\zeta''=\text{id}_{W_0}$. Thus $\mu''$ is an isomorphism. This proves the
theorem.
\end{proof}

\subsection{Local defining equations of $\MPd$ restated}
\lab{localEqre}

\begin{say} The local equation $F=0$ of Theorem \ref{thm:equ1}
admits an elegant form in terms of the terminally weighted tree
$\gamma \in \Lambda_d $ of the associated weighted curves which we now
describe.
\end{say}

\begin{say}
Given a terminally weighted tree $\gamma$, there are three equivalent ways to
describe the local equation near the stratum $\MPd_\gamma$.

The first is in a direct form: to every non-root vertex $a \in \Vg\sta$, we associate 
the coordinate function of $\AA^1$ indexed by $a$:  $\zetaz_a \in
\Gamma (\sO_{\Ao})$. To a terminal vertex
$b \in \Vg^t$, we associate $n$ coordinate functions $w_{b,1},\cdots,w_{b,n}\in\Gamma(\sO_{\Ao})$.
We then  set
$$\Phi_\gamma=(\Phi_{\gamma,1},\cdots,\Phi_{\gamma,n}), \quad
\Phi_{\gamma,e} = \sum_{b\in \Vg^t}z_{[b,o]}w_{b,e},\quad z_{[b,o]}=  \prod_{b \succeq a \succ o} \zetaz_a.
$$
We make a convention that if $\gamma=o$, we define $\Phi_{o,e}=w_e$ and hence
$\Phi_o=(w_1,\cdots,w_n)$.

The second is by induction on $\gamma \ne o$.
{\sl If $a$ is a terminal vertex, we set $\Phi_{a,e}= w_{a,e}$ with $w_{a,e}$ as before. If
$\gamma=o[\gamma_1,\cdots,\gamma_j]$ with $\gamma_i \in
\Lambda_{d_i}$ having roots $v_i$, then set
$$\Phi_{\gamma,e}=\zetaz_{v_1}\Phi_{\gamma_1,e}+\cdots+\zetaz_{v_k}\Phi_{\gamma_j,e}.
$$
}

The third is in terms of the bracket representation of $\gamma$.
Each $\Phi_{\gamma,e}$ is derived from $\gamma$ by dropping the
root $o$, replacing each ghost vertex $a$ by its associated
function $\zetaz_a$, replacing each terminal vertex $b$ by $\zetaz_b\cdot w_{b,e}$, replacing ``,'' by ``$+$'',
and replacing ``['' ``]'' by ``('' ``)''. The resulting expression
$\Phi_{\gamma,e}$ is identical to {the one} from the second
method.
\end{say}

\begin{exam}
Take $\gamma=o[a,b[c,d]]$ (see the first tree in Figure 1). The domain of a generic $f$ in $\MPd_\gamma$ has
a  genus-1 ghost component labelled by the root $o$, 
a genus zero  ghost component labelled by $b$,  and three rational tails labelled by $a$, $c$
and $d$, respectively. The tails $a$ are attached to the genus 1
ghost component. Tails $c$ and $d$ are attached to the rational
component  $b$. By the first approach,
$$\Phi_{\gamma,e} =\zetaz_aw_{a,e}+ \zetaz_b\zetaz_cw_{c,e} \zetaz_b+  \zetaz_b\zetaz_dw_{d,e} \zetaz_b.
$$
By the second,
$$\Phi_{\gamma,e} =\zetaz_aw_{a,e}+\zetaz_b\Phi_{b[c,d]} = \zetaz_aw_{a,e}+\zetaz_b(\zetaz_cw_{c,e}+\zetaz_dw_{d,e}).$$
By the third,
$$\Phi_{\gamma,e} =\zetaz_aw_{a,e}+\zetaz_b(\zetaz_cw_{c,e}+\zetaz_dw_{d,e}).$$
\end{exam}

\begin{say}
We now describe the local model of the singularity type of $\MPd$ near $\MPd_\gamma$.
We let
$$V_\gamma=\prod_{a\in \Vg\sta} \AA^1\cong \AA^h
\and E_\gamma=V_\gamma\times (\prod_{b\in \Vg^t}\Ao)^{\times n}\cong \AA^{h+n\ell},
$$
where $h$ (resp. $\ell$) is the cardinality of $\Vg\sta$ (resp. $\Vg^t$).
The expressions $\Phi_{\gamma,e}$ then become
regular functions on $E_\gamma$ after we identify $z_a$ with the coordinate function of the $a$-th
copy of $\prod_{a\in \Vg\sta} \Ao$ and identity $w_{b,e}$ with the coordinate function of
the $b$-th copy of $\prod_{b\in \Vg^t}\Ao$
in the $e$-th component of the product $(\cdot )^{\times n}$.

We define \beq\lab{zgamma} Z_\gamma = \{(z_{
a},w_{{b},e}) \in E_\gamma \mid \Phi_{\gamma,e}(z,w)=0,\
1\leq e\leq n\}. \eeq
We then define the type $\gamma$ loci in $Z_\gamma$ to be
$$Z_\gamma^0=\{(z,w)\in Z_\gamma\mid z_{ a}=0 \ \text{for all {$a\in \Vg\sta$}}\}.$$
\end{say}

\begin{defi}
We say a DM-stack $S$ has singularity type $\gamma$ at a closed point $s\in S$ if there is
a scheme $y\in Y$ and two smooth morphisms $q_1: Y\to S$ and $q_2: Y\to Z_\gamma$
such that $q_1(y)=s$ and $q_2(y)\in Z_\gamma^0$.
\end{defi}

\def\un{^{\oplus n}}

We have

\begin{theo}\lab{thm:localEquTree}
The stack $\MPd$ has singularity type $\gamma$ along $\MPd_\gamma$.
\end{theo}

\begin{proof}
Let $[u,C]\in\MPd$ be a closed point with associated terminally weighted rooted tree $\gamma$.
We let $\cU=\cV\times_{\fD_1}U\to\cV$ be as in Theorem \ref{thm:equ1}.
Theorem \ref{thm:key} provides trivializations
$$\rho\lsta\sL(\cA) 
\cong \bl \Oplus_{b\in \Vg^t}\sO_\cV\br\oplus \sO_{\cV}^{\oplus (d-\ell+1)}
$$
so that the restriction homomorphism
$$r: \rho\lsta\sL(\cA)\lra \rho\lsta(\sL(\cA)|_{\cA})
$$
is given by
$$\Oplus_{b\in \Vg^t}\zeta_{[b,o]}\oplus 0: \bl\Oplus_{b\in \Vg^t}\sO_\cV\br \oplus \sO_\cV^{\oplus (d-\ell+1)}
\lra\sO_\cV.
$$

The composite homomorphism
$$\rho\lsta\sL(\cA)\un\mapright{\cong} \bl\bl\Oplus_{b\in \Vg^t}\sO_\cV\br\oplus \sO_\cV^{\oplus (d-\ell+1)}\br\un
\mapright{\text{pr}} \bl \Oplus_{b\in \Vg^t}\sO_\cV\br\un
$$
induces a morphism
$$\cE_\cV\lra \cV\times \bl\prod_{b\in\Vg^t}\Ao\br^{\times n};
$$
the regular functions $\zeta_a$ define a  morphism \beq
\lab{thephi} \phi=\prod_{a\in\Vg\sta}\zeta_a : \cV\lra \bl
\prod_{a\in\Vg\sta}\Ao\br=V_\gamma. \eeq
Together, they define a
morphism
\beq \lab{thephitilde} \tilde\phi: \cE_\cV\lra \cV\times
\bl\prod_{b\in\Vg^t}\Ao\br^{\times n}\mapright{}E_\gamma=
V_\gamma\times (\prod_{b\in\Vg^t}\Ao\br^{\times n}.
\eeq
We comment
that since deformations of nodal curves are unobstructed,
the morphisms $\phi$ and $\tilde\phi$ are smooth.

By Theorem \ref{thm:key}, 
\beq\lab{pullback}
\tilde\phi\sta(\Phi_\gamma)=F.
\eeq
This proves that $$\tilde\phi |_{(F=0)}: (F=0) \lra Z_\gamma$$ is smooth, since $\tilde\phi$ is smooth.

Finally, because $\cU\to\MPd$ is smooth and $\cU\to (F=0)\sub\cE_\cV$ is an open immersion,
and thus smooth, the composite
\beq\lab{UU}\cU\lra (F=0)\sub\cE_\cV\lra Z_\gamma\sub E_\gamma
\eeq
is smooth. Also,  it is clear that
a lift $\xi\in\cU$ of $[u]\in\MPd_\gamma$ is mapped to a point in $Z_\gamma^0$.
This proves that $\MPd$ has singularity type $\gamma$ at $[u, C]$.
\end{proof}

\subsection{A stratification of a blowing up of $Z_\gamma$}

\begin{say}
For the purposes of keeping track of blowups of $\MPd$,
we need to
classify the singularity types of the blowups of $Z_\gamma$.
Such types will be classified by simple weighted rooted trees that
are monoidal transformations and collapsings of $\gamma$.
\end{say}

\begin{say} We first classify the singularity types of the space $Z_\gamma$
for a simple terminally weighted tree $\gamma$.
The singularity type of a $x\in Z_\gamma$ is defined by its associated tree $\gamma_x$.
Let $x\in Z_\gamma$ (resp. $\in V_\gamma$) and let $x=(z_a,w_b^j)$ (resp. $x=(z_a)$) be
its coordinate representation. The non-vanishing of $z_a$ identifies a subset of $\Vg\sta$:
$$ I_x=\{ a\in \Vg\sta\mid z_a\ne 0\}.
$$
We let $\gamma_x$ be the collapsing of $\gamma$ at vertices in $I_x$.
\end{say}


\begin{lemm}
The scheme $Z_\gamma$ has singularity
type $\gamma_x$ at $x\in Z_\gamma$. 
\end{lemm}

\begin{proof}
This is a direct check.
\end{proof}

\begin{say}
We now investigate the blowing up of $Z_\gamma$. 
Let
$$\mathring\Pi_{\gamma,e}=\{ x\in V_\gamma\mid  \text{the root of $\gamma_x$ is the branch point
and}\; {\bran}(\gamma_x)=k \}.
$$
It is clear that $\mathring\Pi_{\gamma,e}$ is smooth and locally closed. In general, the closure
$\Pi_{\gamma,k}$ of $\mathring\Pi_{\gamma,k}$ in $V_\gamma$ is quite complicated. However, in the
case when ${\bran}(\gamma)\geq k$, $\Pi_{\gamma,k}$ is smooth and bears a simple description.
\end{say}

\begin{lemm}\lab{justK}Let $\ga$ be a simple tree and $k\ge2$. 
If $\bran(\ga)=k$, then $\Pi_{\ga,k}$ consists of all $x\in V_{\ga}$
such that $\bran(\ga_x)=k$. If $\bran(\ga)=0$ or $\bran(\ga)>k$, then $\Pi_{\ga,k}=\emptyset$.
\end{lemm}

\begin{proof}
Let $o\prec v_1\prec \cdots\prec v_r$ be the trunk of $\gamma$ and let $a_1,\cdots, a_k$ be the
direct-descendants of the  branch point $v_r$. Then $\mathring\Pi_{\gamma,k}$
consists of those $x=(z_a)$ so that
$$z_{v_1}\ne 0, \cdots, z_{v_r}\ne 0; \ z_{a_1}=0,\cdots, z_{a_k}=0.
$$
Its closure is given by the vanishing of all $z_{a_j}$:
\beq\lab{Pi}
\Pi_{\gamma,k}=\{\zetaz_{a_1}=\cdots=\zetaz_{a_k}=0\}.
\eeq
By the definition of $\gamma_x$,
$x\in \Pi_{\gamma,k}$ if and
only if ${\bran}(\gamma_x)=k$.
The second claim is clear.
\end{proof}

\begin{say}\lab{induc-blowup}
We now consider the blowup of $V_\gamma$ along $\Pi_{\gamma,k}$ in the case $\bran(\gamma)=k$.
We let
$$\gamma=\overline{o\, v_r}[\gamma_1,\cdots,\gamma_k]=
o[v_1[\cdots[v_r[\gamma_1,\cdots,\gamma_k]]\cdots]],
$$
and let $a_i$ be the root of $\gamma_i$; thus, $a_1,\cdots,a_k$ are the direct descendants of the
branch point $v_r$ of $\gamma$.
Accordingly,
\beq\lab{Phi}
\Phi_{\gamma,e}(z,w)=(z_{v_1}\cdots z_{v_r})(z_{a_1}\Phi_{\gamma_1,e} +\cdots+z_{a_k}\Phi_{\gamma_k,e}) .
\eeq

We denote the blowup of $V_\gamma$ along $\Pi_{\gamma,k}$ by $V_{\gamma,[k]}$ and define
$$Z_{\gamma,[k]}=Z_\gamma\times_{V_\gamma} V_{\gamma,[k]}.
$$
The scheme $Z_{\gamma,[k]}$ is a subscheme of $Z_\gamma \times\PP^{k-1}$ defined by the equations
$\zetaz_{a_j} u_i = \zetaz_{a_i} u_j$ for $1 \le i, j \le k$.
(Here $[u_1, \cdots, u_{k}]$ is the homogeneous coordinate on
${\mathbb P}^{k-1}$.) In the affine open subset $\{u_i=1\}$,  we have
$\zetaz_{a_j} = \zetaz_{a_i} u_{j}$ for $ j \ne i$. Thus, over this chart,
$Z_{\gamma,[k]}$ is defined by 
\beq\lab{Phi1} (z_{v_1}\cdots z_{v_r})\cdot \zetaz_{a_i}
(\Phi_{\gamma_i,e} + \sum_{j \ne i} u_j \Phi_{\gamma_j,e}) =0,
\quad 1 \le e \le n. \eeq

There are two cases. If $\gamma_i$ is a single-vertex tree,
then \eqref{Phi1} becomes \beq\lab{Phi2} (z_{v_1}\cdots
z_{v_r}\zetaz_{a_i}) (w_{a_i,e} + \sum_{j \ne i} u_j
\Phi_{\gamma_j,e} )=0, \quad 1 \le e \le n.
\eeq
After introducing
$\tilde w_{a_i,e}=w_{a_i,e} + \sum_{j \ne i} u_j
\Phi_{\gamma_j,e}$,  \eqref{Phi2} becomes
\beq
\lab{Phi22} z_{v_1}\cdots z_{v_r}\zetaz_i \tilde w_{a_i,e}=0,
\quad 1 \le e \le n.
\eeq
This is the system associated to the
tree $\gamma'=\overline{o\, v_r}[a_i]$, a path tree containing
vertices $o,v_1,\cdots, v_r, a_i$. It is also the
advancing $a_i$ in $\gamma$ and thus is in $\Mg$. This shows that
$Z_{\gamma, [k]} \cap \{u_i=1\} \cong Z_{\gamma'}$.

If $\gamma_i$ is a nontrivial rooted tree and is of the form
$\gamma_i=a_i [\gamma'_1,\cdots,\gamma_m']$, where $m>1$ because $\gamma$ is
simple by assumption, the system \eqref{Phi1}
becomes
\beq\lab{Phi4} (z_{v_1}\cdots z_{v_r})\cdot \zetaz_{a_i}
(\sum_{s=1}^{m} \zetaz_{b_s} \Phi_{\gamma_s',e}+ \sum_{j \ne i}
u_j \Phi_{\gamma_j,e}) =0,\quad 1\leq e\leq n,
\eeq
where $b_s$ is the
root of $\gamma_s'$. After replacing $u_j$ by $z_{b_j}$, (noticing
that the system \eqref{Phi4} does not contain the variables
$z_{a_j}$,) this is the system associated to the tree
$\gamma'$ that is obtained from $\gamma$ by advancing the vertex
$a_i$. 
This shows that $Z_{\gamma, [k]} \cap
\{u_i=1\} \cong Z_{\gamma'}$.
\end{say}

The above yields the following statement.

\begin{lemm}\lab{Zgk}  If $\ga$ is a simple terminally weighted tree such that 
$\bran(\ga)=k$, then the blowup $Z_{\ga,[k]}$ of $Z_{\ga}$ can be
covered by open subsets isomorphic to $Z_{\ga'}$ with 
$\gamma\in {\text Mon}(\ga)$.
\end{lemm}


\begin{say}
We remark here that in the case of 
\eqref{Phi22}, the vanishing locus  $$\tilde{w}_{a_i,1} \cdots
\tilde{w}_{a_i,n} \cdot z_{v_1}\cdots z_{v_r} =0$$  has normal
crossing singularities.
\end{say}

\begin{exam}
Consider $\gamma=o[a,b[c,d]]$.  Then we have
$$\Phi_{\gamma,e}=z_aw_{a,e}+z_b(z_cw_{c,e}+z_dw_{d,e}), \quad  1 \le k \le n.$$
We blow up $Z_\gamma$ along the locus $\{z_a=z_b=0\}.$ The
blown-up  is a subspace of $Z_\gamma \times {\mathbb P}^1$
defined by equations $z_a u_b = z_b u_a$, where $[u_a, u_b]$ are
the homogeneous coordinates of ${\mathbb P}^1$. In the affine open
subset $u_a=1$, the equation $\Phi_{\gamma,e}=0$ becomes
$$ z_a (w_{a,e} + u_b (z_cw_{c,e}+z_dw_{d,e}))=0;
$$
it has normal crossing singularities, and the resulting
system is associated to the tree $o[a]$, the advancing of $a$ of
$o[a,b[c,d]]$. In the affine open subset $u_b=1$, the equation
$\Phi_{\gamma,e}=0$ becomes
$$ z_b (u_a w_{a,e} + z_cw_{c,e}+z_dw_{d,e})=0;$$
this is the system associated to the tree $o[b[a,c,d]]$, the
advancing of $b$ of $o[a,b[c,d]]$.
\end{exam}

\subsection{Local equations of $\wMPd$}   


\begin{say} We begin with recalling the notations and facts
about the blown-up $\wMPd$. In \ref{theta}, we introduced
$\mathring\Theta_k$ that is a
smooth locally closed substack of $\fMw$ of ghost core elliptic
curves attached on $k$ distinct smooth points with $k$-rational tails;
$\Theta_k$ is the closure of $\mathring\Theta_k$. In
\ref{inductive-blowups}, we successively blow up $\fMw$ along proper transforms of
$\Theta_k$.
Inductively, after obtaining $\fM\uwt_{1,[k-1]}$, we blow it up
along the proper transform $\Theta_{k,[k-1]} \subset
\fM\uwt_{1,[k-1]}$ of $\Theta_k \subset \fM\uwt_1$.
We denote by $\Theta_{i, [k]} \sub \fM\uwt_{1,[k]}$ the proper transform of $\Theta_i$.
We let $\fE_{[k]}\sub \fM\uwt_{1,[k]}$ be the exceptional divisor of the $\fM\uwt_{1,[k]}\to
\fM\uwt_{1,[k-1]}$.
Finally, we denote $\wfMw$ the resulting limit stack.

As the image of $\MPd\to\fMw$ is disjoint from $\Theta_k$ for $k>d$, the
fiber product $\wMPd$ can be defined after $d$-th blowing up of $\fMw$:
$$\wMPd = \MPd \times_{\fMw} \fM\uwt_{1,[d]}.
$$
To pave a way for our proof, we also need to record the
intermediate blowup spaces. For this, we introduce for $k \ge 1$
$$\wMPd_{[k]} = \MPd \times_{\fMw} \fM\uwt_{1,[k]}.
$$
\end{say}

\begin{say} Recall that $\Lambda_d$ is the index set for the canonical
stratification $\MPd=\bigcup_{\gamma \in \Lambda_d} \MPd_\gamma$.
We set $\Lambda_{d, [1]}=\Lambda_d$ and defined $\Lambda_{d, [k]}$
inductively for $k \ge 2$ in \ref{kmonTrees}.
\end{say}


\begin{lemm}\lab{adhoc}
To each closed point $s\in\widetilde{M}_1(\Pn,d)_{[k]}$ we can find a graph $\gamma\in\Lambda_{d,[k]}$
and a smooth morphism $q_{\gamma,1}:W_\gamma\to \widetilde{M}_1(\Pn,d)_{[k]}$ whose image
contains $s$ of which the following holds: \\
(i). 
there are smooth chart $\cV_{\gamma}\to\fM_{1,[k]}\uwt$
and smooth morphisms $q_{\gamma,2}$, $\phi_\gamma$, $\psi_\gamma$, and $p_\gamma$ shown
below,  making the diagram commutative
$$ 
\begin{CD}
Z_\gamma @<{q_{\gamma,2}}<< W_\gamma @>{q_{\gamma,1}}>> \widetilde{M}_1(\Pn,d)_{[k]}\\
@VV{\psi_\gamma}V @VV{p_\gamma}V @VVV\\
V_\gamma @<{\phi_\gamma}<< \cV_\gamma @>>> \fM_{1,[k]}\uwt;
\end{CD}
$$ 
(ii) for any $i\geq k+1$, 
$$ 
W_{\ga}\times_{\fM_{1,[k]}\uwt}\mathring\Theta_i = 
W_{\ga}\times_{V_{\ga}}\mathring\Pi_{\ga,i}.
$$ 
\end{lemm}

\begin{proof} The proof is by induction on $k$. 
In the case $k=1$, the desired morpisms are provided
in the proof of Theorem \ref{thm:localEquTree}.
Suppose $k\ge2$ and the lemma holds for a closed point $s\in \tilde{M}_1(\Pn,d)_{[k-1]}$.
Let $W_\gamma$ etc., be the corresponding data provided by the statement of the
lemma.  

If $\bran(\ga)=0$ or $\bran(\ga)>k$, then $\ga\in\Lambda_{d,[k+1]}$, while 
by (ii) with $k$ replaced by $k-1$ and Lemma 5.12
$$W_{\ga}\times_{\fM_{1,[k-1]}\uwt}\Theta_k= \eset.
$$
Thus, $s$ does not lies over the blowing up center of $\fM_{1,[k-1]}\uwt$, and
is in $\widetilde{M}_1(\Pn,d)_{[k]}$. By shrinking
$W_\gamma$ if necessary, the morphism $q_{\gamma,1}$ (provided by the inductive assumption)
lifts to $W_\gamma\to \widetilde{M}_1(\Pn,d)_{[k]}$
and the $k$-version of the statements (i) and (ii) are identical to its $(k-1)$-version.

If $\bran(\ga)=k$, we define $\cV_{\gamma,[k]}$ to
be the fiber product using the right square of \eqref{square-3} (see below). 
By the property (ii) and the construction of the
blowups $\fM_{1,[k]}\uwt\to\fM_{1,[k-1]}\uwt$ and $V_{\gamma,[k]}\to V_\gamma$,
there is a unique $\phi_{\gamma, [k]}$ making
the left square a fiber product:
\beq\lab{square-3}
\begin{CD}
V_{\gamma,[k]} @<{\phi_{\gamma,[k]}}<< \cV_{\gamma,[k]} @>>> \fM_{1,[k]}\uwt\\
@VVV @VVV @VVV\\ 
V_\gamma @<\phi_\gamma<< \cV_\gamma @>>> \fM_{1,[k-1]}\uwt
\end{CD}
\eeq
This shows that
$$W_\gamma\times_{\fM_{1,[k-1]}\uwt} \fM_{1,[k]}\uwt=W_\gamma\times_{V_\gamma} V_{\gamma,[k]},
$$
which is smooth over $Z_{\gamma,[k]}=Z_\gamma\times_{V_\gamma} V_{\gamma,[k]}$.
Therefore, by Lemma 5.14
$W_{\ga}\times_{\fM_{1,[k-1]}\uwt} \fM_{1,[k]}\uwt$
is covered by smooth morphisms 
$$W_{\ga}\times_{Z_{\ga}}Z_{\ga'}\lra
W_{\ga}\times_{\fM_{1,[k-1]}\uwt} \fM_{1,[k]}\uwt
$$
with $\ga'\in\Mon(\ga)\subset\Lambda_{d,[k]}$ that satisfy the statements of the lemma. Finally,
because 
$$\widetilde{M}_1(\Pn,d)_{[k]}=\widetilde{M}_1(\Pn,d)_{[k-1]}\times_{\fM_{1,[k-1]}\uwt} \fM_{1,[k]}\uwt,
$$
the lemma follows.
\end{proof}

\begin{say} We now describe a stratification of $\wMPd$. We define an equivalence relation
on $\Lambda_{d,[k]}$ by demanding that $\gamma\sim \gamma'$ if 
$\gamma$ is isomorphic to $\gamma$ as rooted but unweighted trees. Note that $W_\gamma$ 
and $W_{\gamma'}$ has isomorphic germs at their origin if and only if $\gamma\sim
\gamma'$. For $\gamma\in\Lambda_{d,[k]}$, we denote by $[\gamma]$ the equivalence
class of $\gamma$ in $c$.

For any point $\xi \in \wMPd_{[k]}$, we let
$$(\eta, W_\ga) \lra (\xi ,\wMPd_{[k]}) \and (\eta, W_\ga) \lra (x,
Z_\gamma)$$ be the 
{smooth} morphisms
provided by the previous lemma. We define $[\gamma_\xi]= [\gamma_x]$.
By the comment in the previous paragraph, the equivalence class $[\gamma_\xi]$ is
independent of the
choice of the chart covering $\xi$. Therefore, we can define
$$\wMPd_{[\gamma], [k]} = \{\xi \in \wMPd_{[k]} \mid [\gamma_\xi]=[ \gamma]\}.$$
In general, $\wMPd_{[\gamma], [k]}$ is  a disjoint union of closed substacks.
\end{say}

It follows immediately from the previous lemma that

\begin{lemm}\lab{thm:induction}
The stack $\wMPd_{[k]}$  has a stratification $\coprod
\wMPd_{[\gamma],[k]}$ indexed by $[\gamma]\in\Lambda_{d, [k]}/\sim$
such that it has singularity type $[\gamma]$ along stratum $\wMPd_{[\gamma],[k]}$.
\end{lemm}

By Lemma \ref{mon2}, $\Lambda_{d, [d]}$ consists of only path
trees. Note that in this case, $\gamma\sim \gamma'$ if and only if $\gamma = \gamma'$.
  Hence we have

\begin{theo}\lab{thm:inductionCase=d}
$\wMPd$  has a stratification $\coprod_{\gamma\in \Lambda_{d,[d]}} \wMPd_{\gamma}$ and
 has normal crossing singularity type
$\gamma$ along $\wMPd_{\gamma}:=\wMPd_{[\gamma],[d]}$.
\end{theo}

\subsection{Proof of Theorem \ref{thm:equ3}} \begin{proof} For each $\gamma \in \Lambda_{d,
[d]}$, it follows from the proof of Lemma \ref{adhoc} that 
to each $\xi\in \wMPd_{\gamma}$, we can find an \'etale $\xi\in
\widetilde U_\gamma \to \wMPd$, a scheme $\widetilde \cU_\gamma$
and a smooth morphism $q_1: \widetilde \cU_\gamma \to \widetilde
U_\gamma$ with an open embedding $q_2: \widetilde\cU_\gamma  \lra
\bl \widetilde F =0 \br \subset \cE_{\widetilde\cV}$. Let $\bl
\widetilde F =0 \br \to Z_\gamma \sub E_\gamma$ be induced by
\eqref{UU}. 
By the same reasoning as in the proof of Theorem
\ref{thm:localEquTree}, if we let
$$\Phi_\gamma=(\Phi_{\gamma,1},\cdots,\Phi_{\gamma,n})\in\Gamma(\sO_{E_\gamma}\un),
$$ and $\tilde\phi: \cE_{\widetilde\cV} \lra
E_\gamma$ be the natural morphism (similarly defined as
\eqref{thephitilde}), then
$$ \tilde\phi\sta(\Phi_\gamma)=\widetilde F.$$ Since $\gamma \in
\Lambda_{d, [d]}$, $\gamma$ is of the form $o[v_1[\cdots[v_r]]]$
with $ r \le d$. Hence each 
$$\Phi_{\gamma,e} = w_e z_1\cdots
z_r.
$$ 
This proves the theorem.
\end{proof}

\begin{say} {\bf Proof of Theorem \ref{thm:m1}}
\begin{proof}
Theorem \ref{thm:inductionCase=d} is a refined version of Theorem
\ref{thm:m1}.
\end{proof}

Finally, since the primary component $\overline{M}_1(\Pn,d)_0$ of $\overline{M}_1(\Pn,d)$
is irreducible and of dimension $(n+1)d$, thus the primary component
$\wMPd_0$ of $\wMPd$, which is the proper transform of $\overline{M}_1(\Pn,d)_0$, 
is smooth and is defined by $(w_1=\cdots=w_n=0)$ in each chart.
\end{say}

\subsection{Proof of Theorem \ref{thm:m2}}
\vskip3pt

We now  prove

\medskip\noindent
{\bf Theorem \ref{thm:m2}.} {\it {For any $r\geq 0$, the
direct image sheaf $\tilde{\pi}_{\mu*} \tilde{f}_\mu
\sta\sO_{\Pn}(r)$ is locally free over every  component
$\wMPd_\mu$ where $\mu$ is either 0 or a partition of $d$. It is
of rank $rd$ when $\mu=0$ and of rank $rd+1$ otherwise.}}

\begin{proof}
 For $r \ge 1$, we set $m=dr$ and follow the notations introduced earlier.
 We let $\tilde\xi\in \wMPd_\mu$ be any point and let $\xi\in \MPd$ its image.
We pick a local chart $\cU= \cV \times_{\fD_1} U$ of $\xi\in\MPd$, and introduce
$$\eta: \widetilde\cV = \cV \times_{\fD_1} \widetilde{\fD}_1 \lra \cV
,  \quad \lambda: \widetilde\cU = \widetilde\cV \times_{\fD_1} U
\lra \cU.
$$
Then $\widetilde\cU$ is a chart of $\wMPd$. 

Continue to write $(\cC, \cD)$ the tautological family over $\cV$, we set
$$\widetilde\cX'=\cX \times_U  \widetilde\cU=\cC \times_\cV \widetilde\cU, \quad \widetilde\cD'= \cD \times_U \widetilde\cU= \cD \times_\cV \widetilde\cU.
$$
Here the second equality in each set of identities follows from the universality of $\fD_1$.
We let $\tilde{\cA}'=\cA \times_\cV \widetilde\cU$ and
$\tilde{\sL}' = \sO_{\widetilde\cX'}(\widetilde\cD')$. The square \beq\lab{square2}
\begin{CD}
\widetilde\cX' @>{\tilde\beta}>> \cC \\
@V{\tilde\pi'}VV @V{\rho}VV  \\
\widetilde\cU @>{\beta}>> \cV,
\end{CD} \eeq
combined with the cohomology and base change theorem, gives the following commutative diagram
$$\begin{CD} \tilde\pi'_* \tilde{\sL}'(\tilde{\cA}')
@>{\tilde\varphi}>> \tilde\pi'_* \sO_{\tilde{\cA}'}(\tilde{\cA}') \\
@| @|\\
\beta^* \rho_* \sL(\cA) @>{\beta^*\varphi}>> \beta^* \rho_*
\sO_{\cA}(\cA).
\end{CD}
$$
Thus investigating $\tilde\varphi$ is equivalent to investigating $\beta^*\varphi$.

For $\varphi$,  by \eqref{mainHom} it is the direct sum of the zero homomorphism on $\sO_\cV$ with
$$\Oplus_{i=1}^m \varphi_i:
 \bigoplus_{i=1}^m \rho_*\sM_i\lra \rho_*
\sO_\cA(\cA), \quad \varphi_i= \zeta_{[\delta_i, a]}.
$$
Therefore, 
$\tilde\varphi$ is given by a
direct sum of the zero homomorphism
with the pullbacks of  $\varphi_i= \zeta_{[\delta_i,a]}$.

We now let $\tilde\gamma=o[v_1[\cdots[v_r]\cdots]] \in \Lambda_{d,
[d]}$ be the tree associated to $\tilde\xi$ and let
$\gamma$ be the tree associated to $\xi$.
By Lemma \ref{adhoc}, to $\tilde\xi\in \wMPd_{[\tilde\gamma]}$ we
can find an \'etale neighborhood $ \widetilde U_{\tilde\gamma}  \ni \tilde\xi
\to \wMPd$ and  a scheme $W_{\tilde\gamma}$ with smooth morphisms
$$W_{\tilde\gamma} \stackrel{\theta} \lra  \widetilde U_{\tilde\gamma}\and
W_{\tilde\gamma}\lra Z_{\tilde\ga}.
$$
(We can assume  $\widetilde\cU=W_{\tilde\gamma}$.)
By the inductive construction of \ref{induc-blowup},
we know that the tree $\tilde\gamma$ is obtained from
$\gamma$ by successively advancing
vertices until a terminal vertex is advanced. Assume that $1 \le i
\le \ell$ is the terminal vertex (in $\gamma$) that is advanced in the last step
(here we adopt the indexing scheme as arranged in the proof of
Theorem \ref{thm:key}). Then by following the same inductive
construction of \ref{induc-blowup} step by step, we obtain
$$\beta^*\varphi_i= z_{v_1} \cdots z_{v_r} 
\and \beta^*\varphi_i | \beta^*\varphi_j \;\hbox{for all $j \ne
i$}.$$ Thus as in the proof of Theorem \ref{thm:key}, using a new
basis of
$$\tilde\pi'_* \tilde{\sL}'(\tilde{\cA}') =\beta^* \rho_* \sL(\cA)
=\sO_{W_{\tilde\gamma}}\oplus \bigoplus_{i=1}^m \beta^*\rho_*\sM_i \cong
\sO_{W_{\tilde\gamma}}^{\oplus m+1},$$
 we see that the kernel of $\tilde\varphi$ is a direct sum of $\sO_{W_{\tilde\gamma}}^{\oplus
 m}$ with the kernel of the homomorphism
$$z_{v_1} \cdots z_{v_r}: \sO_{W_{\tilde\gamma}} \lra \sO_{W_{\tilde\gamma}}.$$

Now if $\tilde\xi \in \wMPd_0$, then 
 $z_{v_1} \cdots z_{v_r}$ does not vanish at general points of 
 $\theta^{-1}(\widetilde{U}_{\tilde\gamma} \cap \wMPd_0)$. Hence the 
 kernel sheaf of $\tilde\varphi$ is locally free of rank $m$ over 
 $\theta^{-1}(\widetilde{U}_{\tilde\gamma} \cap \wMPd_0)$.
 If $\tilde\xi \in \wMPd_\mu$ for $\mu$ a partition of $d$, then,  one of $z_{v_1},
\cdots, z_{v_r}$ vanishes along $\theta^{-1}(\widetilde{U}_{\tilde\gamma} \cap \wMPd_\mu)$.
Hence the 
 kernel sheaf of $\tilde\varphi$ is locally free of rank $m+1$ over 
 $\theta^{-1}(\widetilde{U}_{\tilde\gamma} \cap \wMPd_\mu)$.
This proves the theorem.
\end{proof}

\subsection{Remarks on moduli spaces of stable maps with marked points}

 Finally, we point out that all the main results in this paper generalize directly to moduli space
 of genus one stable maps
with marked points. 
This extension requires introducing blowup loci involving the marked points analogous to those described
in \cite{VZ}.

\end{document}